\documentclass[a4paper, 11pt, headings = small, abstract]{scrartcl}
\usepackage[english]{babel}

\usepackage{amscd,amsmath,amsthm,array,hhline, enumitem, mathrsfs, booktabs,fancybox,calc,xcolor,graphicx,bbm,xspace,nicefrac,stmaryrd,url,arcs,listings,stmaryrd,crossreftools,comment,amssymb}
\usepackage[semibold]{libertine}
\usepackage[upint,amsthm]{libertinust1math}

\usepackage{bm}
\usepackage[scr=boondoxupr, scrscaled = 1.0, cal =cm]{mathalpha}
\usepackage[scaled=0.95]{inconsolata}

\DeclareMathAlphabet{\mathsf}{OT1}{\sfdefault}{m}{n}
\SetMathAlphabet{\mathsf}{bold}{OT1}{\sfdefault}{b}{n}
\DeclareMathAlphabet{\mathfrak}{U}{jkpmia}{m}{it}
\SetMathAlphabet{\mathfrak}{bold}{U}{jkpmia}{bx}{it}
\usepackage[utf8]{inputenc}

\usepackage{etoolbox}
\usepackage[normalem]{ulem}
\usepackage{mathtools}
\usepackage{geometry}
\usepackage{float}
\usepackage{tikz}
\usepackage{tikz-cd}
\usepackage[outdir =./]{epstopdf}
\usepackage[citestyle = numeric-comp,bibstyle = numeric, firstinits=true, natbib = true, backend = bibtex,maxbibnames=99, url=false, isbn=false]{biblatex}
\numberwithin{equation}{section}

\usepackage{chngcntr}
\counterwithin{figure}{section}
\counterwithin{table}{section}

\usepackage[colorlinks=true, allcolors=myteal]{hyperref}
\definecolor{grey_pers}{RGB}{69 90 100}
\definecolor{WIMgreen}{RGB}{60 134 132}
\definecolor{red_pers}{RGB}{204 37 41}
\definecolor{UMblue}{RGB}{4 47 86}
\definecolor{myteal}{RGB}{0 123 137}
\definecolor{dartmouthgreen}{rgb}{0.05, 0.5, 0.06}\definecolor{cobalt}{rgb}{0.0, 0.28, 0.67}\definecolor{coolblack}{rgb}{0.0, 0.18, 0.39}
\definecolor{glaucous}{rgb}{0.38, 0.51, 0.71}\definecolor{hooker\'sgreen}{rgb}{0.0, 0.44, 0.0}\definecolor{lemonchiffon}{rgb}{1.0, 0.98, 0.8}\definecolor{oucrimsonred}{rgb}{0.6, 0.0, 0.0}\definecolor{radicalred}{rgb}{1.0, 0.21, 0.37}\definecolor{raspberry}{rgb}{0.89, 0.04, 0.36}\definecolor{royalazure}{rgb}{0.0, 0.22, 0.66}
\definecolor{Purple}{RGB}{103 58 183}
\definecolor{Green}{rgb}{.13 .55 .13}

\theoremstyle{plain}
\newtheorem{theorem}{Theorem}[section]
\newtheorem{proposition}[theorem]{Proposition}
\newtheorem{lemma}[theorem]{Lemma}
\newtheorem{corollary}[theorem]{Corollary}

\theoremstyle{definition}
\newtheorem{definition}[theorem]{Definition}
\theoremstyle{assumption}
\newtheorem{assumption}[theorem]{Assumption}
\theoremstyle{remark}
\newtheorem{remark}[theorem]{Remark}

\setkomafont{sectioning}{\normalcolor\bfseries}
\setkomafont{descriptionlabel}{\normalcolor\bfseries}
\setkomafont{section}{\large}
\setkomafont{subsection}{\normalsize}
\setkomafont{subsubsection}{\normalsize}
\setkomafont{paragraph}{\normalsize}
\setkomafont{subparagraph}{\normalsize}
\setkomafont{author}{\large}
\setkomafont{title}{\large\selectfont\bfseries}
\setkomafont{date}{\normalsize}

\def\supp{\operatorname{supp}}

\def\E{\mathbb{E}}

\def\N{\mathbb{N}}

\def\N{\mathbb{N}}

\def\R{\mathbb{R}}
\definecolor{darkred}{rgb}{0,0.6,0}

\newcommand{\PP}{\mathbb{P}}

\renewcommand{\subseteq}{\subset}

\renewcommand{\hat}{\widehat}
\newcommand{\e}{\mathrm{e}}
\renewcommand{\tilde}{\widetilde}%
\renewcommand{\d}{\mathop{}\!\mathrm{d} }
\newcommand{\lebesgue}{\boldsymbol{\lambda}}

\newcommand{\cost}{\kappa}
\DeclareMathOperator*{\argmin}{arg\,min}
\DeclareMathOperator{\interior}{int}
\DeclareMathOperator{\Vol}{Vol}
\DeclareMathOperator{\tr}{tr}
\newcommand{\mat}[1]{\begin{bmatrix}#1\end{bmatrix}}
\newcommand{\n}[2][]{#1\lvert#2#1\rvert}
\newcommand{\coloneqq}{\coloneq}

\restylefloat{table}

\newcommand*\diff{\mathop{}\!\mathrm{d} }

\newcommand{\one}{\mathbf{1}}

\newcommand{\vertiii}[1]{{\left\vert\kern-0.25ex\left\vert\kern-0.25ex\left\vert #1
    \right\vert\kern-0.25ex\right\vert\kern-0.25ex\right\vert}}

\def\supp{\mathrm{supp}}

\let\originalleft\left
\let\originalright\right
\renewcommand{\left}{\mathopen{}\mathclose\bgroup\originalleft}
\renewcommand{\right}{\aftergroup\egroup\originalright}

\allowdisplaybreaks

\makeatother
\title{\fontsize{16}{19} \selectfont Data-driven rules for multidimensional reflection problems}
\author{S\"oren Christensen\thanks{Kiel University, Department of Mathematics, Heinrich-Hecht-Platz 6,
24118 Kiel, Germany. \newline Email: \href{mailto:christensen@math.uni-kiel.de}{christensen@math.uni-kiel.de}} \and Asbjørn Holk Thomsen\thanks{Aarhus University, Department of Mathematics, Ny Munkegade 118, 8000 Aarhus C, Denmark. \newline Email: \href{mailto:a.holk@math.au.dk}{a.holk@math.au.dk}/\href{mailto:trottner@math.au.dk}{trottner@math.au.dk}} \and Lukas Trottner\footnotemark[2]}
\date{\today}

\bibliography{levy_kontrolle}

\begin{document}
\maketitle
\begin{abstract} 
Over the recent past data-driven algorithms for solving stochastic optimal control problems in face of model uncertainty have become an increasingly active area of research. However, for singular controls and underlying diffusion dynamics the analysis has so far been restricted to the scalar case. In this paper we fill this gap by studying a multivariate singular control problem for reversible diffusions with controls of reflection type. Our contributions are threefold. We first explicitly determine the long-run average costs as a domain-dependent functional, showing that the control problem can be equivalently characterized as a shape optimization problem. For given diffusion dynamics, assuming the optimal domain to be strongly star-shaped, we then propose a gradient descent algorithm based on polytope approximations to numerically determine a cost-minimizing domain. Finally, we investigate data-driven solutions when the diffusion dynamics are unknown to the controller. Using techniques from nonparametric statistics for stochastic processes, we construct an optimal domain estimator, whose static regret is bounded by the minimax optimal estimation rate of  the unreflected process' invariant density. In the most challenging situation, when the dynamics must be learned simultaneously to controlling the process, we develop an episodic learning algorithm to overcome the emerging exploration-exploitation dilemma and show that given the static regret as a baseline, the loss in its sublinear regret per time unit is of natural order compared to the one-dimensional case.
\end{abstract}

{\noindent\small\href{https://mathscinet.ams.org/mathscinet/msc/msc2020.html}{\textit{2020 MSC}}: Primary 93E35, 68T05; secondary 49Q10, 60J60, 62M05 \\
\textit{key words}: Stochastic singular control, reinforcement learning, exploration vs.\ exploitation, reversible diffusions, shape optimization, nonparametric statistics}

\section{Introduction and problem formulation}
\normalfont
Our underlying processes are Langevin diffusions $X$ on $\R^d$ for $d \geq 2$, which is a well studied class of reversible diffusion processes with drift of potential form. For a $\mathcal{C}^2$-function $V\colon\R^d\to\R$ and a $d$-dimensional Brownian motion $W$, $X$ solves the SDE
\[\diff{X_t} =-\nabla V(X_t)\diff{t}+ \sqrt{2}\diff{W_t}.\]

We consider a basic class of stochastic control problems with a clear interpretation: $X$ is interpreted as the position of a particle, which we want to be close to a target state, 0, say. The distance is measured by a locally bounded function $f\colon\R^d\to[0,\infty)$ in such a way that $f(X_t)$ stand for the costs associated with being away from the target state. The decision maker can now control the process by choosing a nonempty bounded domain ($=$ open, connected set) $D\subseteq \R^d$ of class $\mathcal{C}^{2}$ and normally reflecting the process at $\partial D$. We denote the resulting normally reflected processes by $X^D$, which is given as the solution to 
\[\diff X_t^D=-\nabla V(X_t^D)\diff{t}+ \sqrt{2}\diff{W_t}+n(X_t^D)\diff{L_t^D},\]
where $n$ is the unit inward normal vector of $D$ at $x\in\partial D$ and $L^D$ denotes the local time of $X^D$ on $\partial D$, that is, a non-decreasing one-dimensional process with continuous paths that increases only when $X^D\in\partial D$. 
We assume that controlling the process is associated with costs proportional to $L^D$ (with proportionality factor
 $\cost>0$), so that for each $D$, the total costs associated with $D$ until time $t$ are
\[\int_0^tf(X_s^D)\d s+\cost L_t^D.\]
Here, we consider a long-term-average criterion and our problem consists of minimizing
\begin{align}\label{eq:problem}
	J(D)\coloneqq\liminf_{t\to\infty}\frac{1}{t}\E^x\left[\int_0^tf(X_s^D)\d s+\cost L_t^D\right]
\end{align}
over all admissible $D$. The limit is independent of the initial value $x \in D$ due to the ergodic nature of the reflected diffusion, which will be made precise in Theorem \ref{prop:erg_cost}. For known characteristics of the underlying process, the problem we consider here is closely connected to singular stochastic control problems. We discuss this in more detail in Section \ref{subsec:lit} below.

\subsection{Contributions}
In this paper, we first address the question how the solution to the problem \eqref{eq:problem} can be meaningfully characterized when the characteristics of the underlying process are known. Our main contribution here is Theorem \ref{prop:erg_cost} which shows under minimal assumptions that $J(D)$ is the $L^1$-limit of the average costs and is explicitly given by
\begin{align}\label{eq:J_functional}
J(D)=\frac{1}{\int_{D}\e^{-V(y)}\d y}\left(\int_Df(y)\e^{-V(y)}\diff{y} +\cost \int_{\partial D}\e^{-V(y)}\, \mathcal{H}^{d-1}(\diff{y})\right),
\end{align}
where $\mathcal{H}^{d-1}$ denotes the $(d-1)$-dimensional Hausdorff measure. 
The formula is interesting in that the control problem has been transformed into a shape optimization problem.

In Section \ref{sec:numerical} we give a numerical approach to minimizing $J(D)$ based on approximating $D$ by polytopes and then applying a gradient descent method. To this end, we derive explicit formulas (Theorem \ref{theo:numeric}) under the assumption that $D$ is star-shaped and then illustrate the method on problems with underlying Brownian motions and Ornstein--Uhlenbeck processes. 

Finally, in Section \ref{sec:learning} we show that the approach we present provides a suitable basis for addressing the problem in the context of model-based reinforcement learning.  More precisely, we show that the problem can be solved when the drift function $b=-\nabla V$ is unknown to the decision maker, so that the control has to be purely data driven. Our approach is based on estimating the stationary density of the uncontrolled process $\rho$ nonparametrically and specifying the exact rate with techniques from nonparametric statistics for diffusion processes (Theorem \ref{theo:est}). We use this via \eqref{eq:J_functional} to estimate the optimal boundary based on a path of the uncontrolled process and obtain that the resulting static regret has the same sublinear rate (Proposition \ref{prop:stat_regret}). In the more practically relevant situation of simultaneous optimization and data collection, we face an exploration vs. exploitation dilemma. The previous results together with an episodic learning approach lead to a proof of a sublinear regret rate (Theorem \ref{theo:expl} and Corollary \ref{coro:expl}).

\subsection{Related literature}\label{subsec:lit}
Stochastic singular control problems are a class of stochastic control problems that have been extensively studied, see \cite{MR3931325,MR2179357} for textbook treatments. They arise in various applications such as inventory management in operations research, control of queueing networks, portfolio selection with transaction costs in finance, equity issuance in insurance mathematics, position control in engineering, or optimal harvesting in biology. The actions of the agents affect the state and the costs in proportion to the size of the action. This structure implies that the optimal controls usually have the following general form: The space is divided into a region where action is required and a region $D$ where no action is required. The optimal control then reflects the controlled process at the boundary of the no-action region to keep it inside. Therefore, the problem class is closely related to reflected SDEs \cite{sznitman84,pilipenko2014introduction}, as observed for a variate of examples, see e.g., \cite{MR3959083, MR2476430, MR1001925,MR1766432}
(although this connection is difficult to be made precise in full generality, see the discussion in \cite{MR1649007}).
Thus, stochastic singular control can typically be reduced to finding the optimal no-action region $D$ and the optimal reflection direction at the boundary of $D$. However, finding the solution is usually only possible explicitly in a few examples, typically with underlying one-dimensional diffusions, see the following section. 
Characterizing optimal controls becomes much harder when the problem has more than one dimension. Some characterizations of the optimal solutions in special multidimensional cases can be found in \cite{MR4316795, MR4149543, MR4597536, MR1646687,MR1788074}. 
Thus, the problem \eqref{eq:problem} can be viewed as an optimization problem over a class of strategies typically relevant to general singular control problems. We restrict this, however, by the fact that only normal reflections are admissible. This assumption may be justifiable in some cases from the real world problem being modeled. However, we will also see below that in subclasses of problems, such as the radially symmetric case, this is not a restriction at all. In what follows, we will see that the assumption of normal reflection simplifies the problem to the point where a deeper analysis can be performed.

There is a fair amount of literature on the numerical treatment of singular control problems. The methods range from approaches based on a discretization with discrete Markov chains \cite{kushner1990numerical} over an approximation of the solution of the corresponding Hamilton--Jacobi--Bellman equation \cite{kumar2004numerical}, approaches using linear programming \cite{vieten2020solution} up to finite element approximations \cite{vieten2018convergence}. However, all methods have limited applicability, especially in higher dimensions, and care must be taken in the exact implementation. In this paper, we present a new approach that is structurally different from the existing ones.

As described above, the approach we have taken allows for data-driven control when the drift of the process is unknown. This question falls into the currently fast growing field of model based reinforcement learning (RL), where the agent does not know the system parameters and learns them by interacting with the environment and getting feedback. The agent chooses policies based on the current parameter estimation and tries to minimize the regret, which is the gap between the expected reward of the best policy and the actual reward achieved. Many discrete-time RL problems have been studied, where sublinear regret bounds have been obtained for different scenarios, such as bandit problems, tabular Markov decision problems, and linear quadratic (LQ) problems \cite{osband2014model,he2021logarithmic,dean2020sample}. However, for continuous-time RL problems, it is well known that time discretization does not work satisfactorily in standard approaches such as Deep-Q learning \cite{tallec2019making}. Therefore, the concrete model must be included here. Most of the previous works only propose algorithms, and only a few analyze their regrets, mostly for LQ problems. In particular, \cite{duncan1999adaptive} proved an asymptotic sublinear regret for regularized least-squares algorithms in an ergodic continuous-time LQ problem, but without giving the exact order of the bound. Recently, \cite{basei2022logarithmic,guo2021reinforcement} generalized the least-squares algorithms to finite-time horizon episodic settings and gave non-asymptotic regret bounds. These works assume a parametric structure of the problem. On the other hand, \cite{neuman2023statistical} considered propagator models and combined exploration and exploitation schemes to achieve sublinear regrets with high probability. Closest to this paper are the articles \cite{christensen21,christensen21_alt}, in which singular and impulse control problems with a nonparametric statistical diffusion structure are considered. However, these papers make heavy use of the one-dimensional structure of the underlying processes. This leads to the strategy already being described by one or two values, so only these need to be learned. One of the main contributions of this work is that we can give exact sublinear rates of regression even for the case where the optimal strategy is only infinite-dimensional parameterizable. To the best of our knowledge, this work is the first to provide such results in this context.

\section{Optimal reflection as a shape optimization problem for known characteristics}\label{sec:known_drift}
As detailed above, for our approach it is central to find explicit expressions for $J(D)$ for known drift. Before we get to that, we start here with a brief discussion of the one-dimensional special case, since it serves as the main motivation. The main observation for obtaining explicit solutions for all singular control problems for underlying linear diffusions is that the values $J(D)$ can be found (semi-)explicitly in terms of speed measure and scale function. In this linear case, solving \eqref{eq:problem} therefore boils down to a standard optimization problem, see \cite{alvarez2018class,christensen2023two} and the references therein. In the one dimensional ergodic case discussed here, this is realized by the fact that the stationary density of the reflected diffusion is just the conditional density of the uncontrolled diffusion. The key observation in this section is that this also holds in the multivariate case with underlying Langevin diffusions. This turns out to be key to establishing \eqref{eq:J_functional}.

\subsection{Ergodicity of the reflected Langevin diffusions}\label{subsec:langevin}
It is well-known that for an (uncontrolled) Langevin diffusion as introduced above, if  $\e^{-V}$ is integrable, then $X$ has a stationary density given by
\[
	\rho_{\R^d}(x)
	\coloneqq \rho(x) \coloneqq c^{-1}\exp(-V(x)), \mbox{ where }c=c_{\R^d}(V)=\int_{\R^d}\e^{-V(u)}\d u.
\]
For general diffusions, knowing the distribution on the whole state space does not give any information about the stationary distributions of the corresponding diffusions with reflection in a subdomain $D$.
A main observation for our approach is that for Langevin diffusions, this is different.
Indeed, in this particular situation one can obtain the stationary distribution on $D$ by conditioning:
\begin{lemma}\label{lem:inv_dens}
For a bounded domain $D\subseteq \R^d$ of class $\mathcal{C}^{2}$, the density $\rho_D$ given by 
\[
	\rho_{D}(x)
	=c_D^{-1}\exp(-V(x)), \quad x\in D, \mbox{ where }c_{D}=\int_{D}\e^{-V(u)}\d u,
\]
is a stationary density of the normally reflected process $X^D$.
\end{lemma}
This result is well known. For the convenience of the reader, a proof is provided in Appendix \ref{sec:app}.

It will be important for our purposes to have sufficiently fast convergence of the reflected diffusion to equilibrium. To this end,  we will assume that $D \subset \R^d$ is a  bounded domain of class $\mathcal{C}^2$ that is sufficiently nice to guarantee that the Markov process $(X^D, (\PP^x)_{x \in \overline{D}})$ has transition densities $(p_t^D(x,y))_{t > 0, x,y \in \overline{D}}$---that is, $\E^x[f(X^D_t)] = \int_{\overline{D}} p_t^D(x,y) f(y) \diff{y}$ for any bounded measurable function $f \colon \overline{D} \to \R$ and $x \in \overline{D}$---such that for any $t > 0$,  $p_t^D$ is continuous on $\overline{D} \times \overline{D}$ and we have the minorization property 
\begin{equation}\label{eq:minor} 
\inf_{x,y \in \overline{D}} p_1^D(x,y) \geq \delta,
\end{equation}
for some $\delta > 0$, where necessarily $\delta \lebesgue(D) \in (0,1)$. Let us denote by $\mathbf{D}$ the class of bounded domains $D \subset \R^d$ of class $\mathcal{C}^2$ such that these assumptions hold for the process reflected in $\partial D$. These  can be verified under mild assumptions on the boundary $\partial D$ by general results on fundamental solutions of parabolic PDEs with Neumann boundary conditions. For instance, continuity of the transition densities and \eqref{eq:minor} are ensured whenever $\partial D$ is the union of a finite number of hypersurfaces of  class $\mathcal{C}^3$, cf.\ \cite[p.166]{ito92}. 
These assumptions now guarantee uniqueness of the invariant distribution and exponential ergodicity of the process.

\begin{lemma}\label{lem:erg}
Let $D \in \mathbf{D}$ with transition minorization as in  \eqref{eq:minor}. Then, the unique stationary distribution $\pi_D$ of the normally reflected diffusion $X^D$ is given by 
\[\pi_D(\diff{x}) = \rho_D(x) \diff{x}, \quad x \in \overline{D},\]
and $X^D$ is uniformly ergodic, satisfying the bound
\[\lVert P_t^D(x,\cdot) - \pi_D \rVert_{\mathrm{TV}} \leq \tfrac{2}{1-\delta\lebesgue(D)} \mathrm{e}^{\log(1-\delta\lebesgue(D)) t}, \quad x \in \overline{D}, \quad t > 0.
\] 
\end{lemma}
The proof is deferred to Appendix \ref{sec:app}. As a consequence we have the following rate in the ergodic theorem.

\begin{corollary}\label{coro:erg}
Let $D \in \mathbf{D}$. There exists a constant $C=C(D) > 0$ such that for any $h \in L^\infty(\overline{D})$ and $x \in \overline{D}$ it holds 
\[\frac{1}{t}\E^x\Big[\Big\lvert \int_0^t (h(X_s^D) - \pi_D(h)) \diff{s}\Big\rvert\Big] \leq \frac{C \lVert h \rVert_{L^\infty(\overline{D})}}{\sqrt{t}}.\]
\end{corollary} 
Again, the proof can be found in Appendix \ref{sec:app}.

\subsection{Solution of the ergodic control problem}
Given the ergodicity assumptions from the previous subsection we can now fully characterize the ergodic average expected costs $J(D)$ from \eqref{eq:problem} in terms of the invariant distribution of the reflected diffusion $X^D$. In fact, we will show  more: we prove that the bias of the average costs vanishes linearly in time and that their stochastic fluctuation measured in terms of the $L^1$-deviation from $J(D)$ vanishes at square-root rate.

\begin{theorem}\label{prop:erg_cost}
Let $D\in \mathbf{D}$. Then, there exist constants $C(D), C^\prime(D) > 0$ that depend on $D$ but are independent of $x \in \overline{D}$ and $t \geq 1$ such that 
\begin{equation}\label{eq:l1_conv}
\E^x\Big[\Big\lvert \frac{1}{t}\Big(\int_0^t f(X_s^D) \diff{s} + \kappa L^D_t \Big) - \Big(\int_D f(y)\rho_D(y)\d y+\cost\int_{\partial D}\rho_D(y)\, \mathcal{H}^{d-1}(\d y) \Big) \Big\rvert \Big] \leq \frac{C(D)}{\sqrt{t}},
\end{equation} 
and
\begin{equation}\label{eq:exp_conv}
\Big\lvert\E^x\Big[ \frac{1}{t}\Big(\int_0^t f(X_s^D) \diff{s} + \kappa L^D_t \Big) \Big] - \Big(\int_Df(y)\rho_D(y)\d y+\cost\int_{\partial D}\rho_D(y)\, \mathcal{H}^{d-1}(\d y) \Big) \Big\rvert  \leq \frac{C^\prime(D)}{t},
\end{equation}
where $\mathcal{H}^{d-1}$ denotes the $(d-1)$-dimensional Hausdorff measure. 
In particular,
\[
	J(D)
	=\int_Df(y)\rho_D(y)\d y+\cost\int_{\partial D}\rho_D(y)\, \mathcal{H}^{d-1}(\d y).
\]
\end{theorem}
We will need the following basic  result for the proof, which we include here for the lack of a precise reference.
\begin{lemma} \label{lem:uninormal}
	Let $O\subset\R^d$ be a bounded open set of class $\mathcal{C}^k$ for some $k\ge 2$ and let $n$ be the unit inward  normal vector on $\partial O$.
	Then there exists a function $\varphi \in \mathcal{C}^k(\R^d)$ such that $\nabla \varphi = n$ on $\partial O$.
\end{lemma} 

\begin{proof} 
	For a set $A \subset \R^d$ let $d(x,A) \coloneqq \inf\{\lvert x -y \rvert: y \in A\}$ be the distance function to $A$ and denote by $A^\varepsilon \coloneqq \{y \in \R^d: \lvert x-y \rvert < \varepsilon \text{ for some } x \in A\}$  its $\varepsilon$-fattening for $\varepsilon > 0$.
	Let 
	\[
		s(x)\coloneqq
		\begin{cases}
			d(x,\partial O), & x \in O,\\
			-d(x,\partial O), &x \in \R^d \setminus O,
		\end{cases}
	\]
	be the signed distance to the boundary $\partial O$.
	By \cite[Theorem 3]{krantz81} there exists an open neighborhood $U$ of $\partial O$ such that $s \in \mathcal{C}^k(U)$, and in particular $\nabla s$ is well-defined on $\partial O$.
	Moreover, we have $\nabla s=n$ on $\partial O$.
	Indeed,
 for given $y \in \partial O$, by the above there exists $r > 0$ such that $s\vert_{B(y,r)}\colon B(y,r) \to \R$ is a $\mathcal{C}^k$ function and clearly $s\vert_{B(y,r)}^{-1}(\{0\}) = \partial O \cap B(y,r)$. By the smoothness of $\partial O$ we have $s(y+hn(y)) = h$ for $h \geq 0$ small enough, whence
 \begin{equation}\label{eq:dir_der}
		\langle \nabla s(y), n(y)\rangle
		=\nabla_{n(y)}s(y)
		=\lim_{h\to0}\frac{s\big(y+hn(y)\big)-s(y)}{h}
		=1,
\end{equation}
 and therefore $\nabla s(y) \neq 0$. The usual argument via the implicit function  theorem now shows that $\nabla s(y)$ is orthogonal to the hyperplane tangent to $\partial O$ at $y$. That is, $\nabla s(y) \propto n(y)$ and we conclude from \eqref{eq:dir_der} that indeed $\nabla s(y) = n(y)$. 
Now, since $s\in \mathcal C^k(U)$, we must have $s \in \mathcal{C}^k((\partial O)^\varepsilon)$ for some $\varepsilon>0$,
and since $\operatorname{cl}((\partial O)^{\varepsilon/2}) \subset (\partial O)^\varepsilon$ is compact, it  follows from the classical Whitney extension theorem \cite{whitney34} that there exists a function $\varphi \in \mathcal{C}^k(\R^d)$ such that $\varphi = s$ on $\operatorname{cl}((\partial O)^{\varepsilon/2})$.
In particular, on $\partial O$ it holds $\nabla \varphi = \nabla s = n$, and we conclude that $\varphi$ has the desired properties.
\end{proof}
\begin{proof}[Proof of Theorem \ref{prop:erg_cost}]
We only prove the claim on $L^1(\PP^x)$-convergence at square root rate in \eqref{eq:l1_conv}, the statement on convergence in expectation at linear rate in \eqref{eq:exp_conv} follows from similar, but easier considerations. We first have as in the proof of Lemma \ref{lem:inv_dens}, letting $A$ be the differential operator from \eqref{eq:diff_op},
	\begin{align*}
		\int_{D}Af(x)\, \pi_D(\diff{x})
		&=-\int_{\partial D}\langle \nabla f(x), n(x)\rangle\rho_D(x)\, \mathcal{H}^{d-1}(\diff{x}),\quad f\in \mathcal{C}^2(\R^d).
	\end{align*}
	In particular, since $D$ is of class $\mathcal{C}^2$, by Lemma \ref{lem:uninormal} we may choose $\varphi\in \mathcal{C}^2(\R^d)$ such that $\nabla\varphi(x)=n(x)$ for $x\in\partial D$, whereby we find 
	\[
		\int_DA\varphi(x)\, \pi_D(\diff{x})
		=-\int_{\partial D}\rho_D(x)\, \mathcal{H}^{d-1}(\diff{x}).
	\]
	Now by Itô's formula for $t\ge0$ almost surely,
	\begin{align*}
		\varphi(X_t^D)-\varphi(X_0^D) &=\int_{0}^{t}A\varphi(X_s^D)\diff{s}+\int_{0}^{t}\langle\nabla\varphi(X_s^D), n(X_s^D)\rangle \diff{L_s^D}+\int_{0}^{t}\langle \nabla\varphi(X_s^D),\diff{W_s}\rangle \\
		&=\int_{0}^{t}A\varphi(X_s^D)\diff{s}+L_t^D+\int_{0}^{t}\langle \nabla\varphi(X_s^D), \diff{W_s}\rangle,
	\end{align*}
	where we used 
	\[
		\int_{0}^{t}\langle \nabla\varphi(X_s^D), n(X_s^D)\rangle \diff{L_s^D}
		=\int_{\{s\le t:X_s^D\in\partial D\}}\langle \nabla\varphi(X_s^D), n(X_s^D) \rangle \diff{L_s^D}
		=\int_{\{s\le t: X_s^D\in\partial D\}}\diff{L_s^D}
		=L_t^D.
	\]
	Note also that since $\nabla\varphi$ is continuous, it is bounded on $D$, and hence $\int_{0}^{\cdot}\langle \nabla\varphi(X_s^D), \diff{W_s}\rangle$ is a $L^2$-martingale. Combining these, we find for any $x \in \overline{D}$,
\begin{align*}
&\mathbb{E}^x\Big[\Big\lvert \frac{1}{t} L_t^D - \int_{\partial D} \rho_D(x) \, \mathcal{H}^{d-1}(\diff{x}) \Big\rvert \Big]\\ 
&\,= \mathbb{E}^x\Big[\Big\lvert \frac{1}{t}\Big(\varphi(X_t^D)-\varphi(X_0^D)-\int_{0}^{t}(A\varphi(X_s^D) - \pi_D(A\varphi)) \diff{s} - \int_0^t \langle \nabla \varphi(X_s^D), \diff{W_s} \rangle \Big)\Big\rvert\Big] \\
&\,\leq 2\frac{\lVert \varphi \rVert_{L^\infty(D)}}{t} + \E^x\Big[\frac{1}{t} \Big\lvert \int_{0}^{t}(A\varphi(X_s^D) - \pi_D(A\varphi)) \diff{s} \Big\rvert  \Big] + \frac{1}{t} \E^x\Big[\Big\lvert \int_0^t \langle \nabla \varphi(X_s^D), \diff{W_s} \rangle \Big\rvert \Big] \\
&\, \lesssim \frac{\lVert \varphi \rVert_{L^\infty(D)} + \lVert A\varphi \rVert_{L^\infty(D)}}{\sqrt{t}} + \frac{1}{t} \E^x\Big[\Big\lvert \int_0^t \langle \nabla \varphi(X_s^D), \diff{W_s} \rangle \Big\rvert \Big] \\ 
&\, \leq \frac{\lVert \varphi \rVert_{L^\infty(D)} + \lVert A\varphi \rVert_{L^\infty(D)}}{\sqrt{t}} + \frac{1}{t} \Big(\E^x\Big[\int_0^t \lvert \nabla \varphi(X_s^D) \rvert^2 \diff{s} \Big]\Big)^{1/2} \\
&\leq \frac{\lVert \varphi \rVert_{L^\infty(D)} + \lVert A\varphi \rVert_{L^\infty(D)}+ \lVert \lvert \nabla \varphi \rvert \rVert_{L^\infty(D)}}{\sqrt{t}} 
\end{align*}
where we used Corollary \ref{coro:erg} for the second inequality and Hölder inequality together with It\^{o}-isometry for the third inequality. 
 Using this, another application of Corollary \ref{coro:erg} for the continuous cost component yields 
 \[\E^x\Big[\Big\lvert \frac{1}{t}\Big(\int_0^t f(X_s^D) \diff{s} + \kappa L^D_t \Big) - \Big(\int_Df(y)\rho_D(y)\d y+\cost\int_{\partial D}\rho_D(y)\, \mathcal{H}^{d-1}(\d y) \Big) \Big\rvert \Big] \leq \frac{C}{\sqrt{t}},\]
 as claimed.
\end{proof}

Putting pieces together, we obtain a formula for $J(D)$ which is just based on the stationary density $\rho$ of the uncontrolled process if the latter is ergodic.

\begin{corollary}\label{coro:J_stat}
For any $D \in \mathbf{D}$, it holds that
\[J(D)=\frac{1}{\int_{D}\e^{-V(y)}\d y}\left(\int_Df(y)\mathrm{e}^{-V(y)}\d y+\cost \int_{\partial D}\mathrm{e}^{-V(y)}\, \mathcal{H}^{d-1}(\diff{y})\right).
\]
If, moreover, $\exp(-V) \in L^1(\R^d)$, then 
\[J(D)=\frac{1}{\int_{D}\rho(y)\d y}\left(\int_Df(y)\rho(y)\d y+\cost \int_{\partial D}\rho(y)\, \mathcal{H}^{d-1}(\diff{y})\right).
\]
\end{corollary}
We point out that this result is a multidimensional version of \cite[Lemma 2.1]{alvarez2018class}.

\section{Numerical optimization}\label{sec:numerical}
Corollary \ref{coro:J_stat} shows that, for known $\cost>0,V$ and $f$, our problem for known dynamics boils down to minimizing the functional
\[
	J\colon D\mapsto\frac{1}{\int_{D}\e^{-V(x)}\d x}\left(\int_Df(x)\e^{-V(x)}\,\d x+\cost\int_{\partial D}\e^{-V(x)}\, \mathcal{H}^{d-1}(\d x)\right)
\]
over a suitable set of bounded domains $D\subseteq \R^d$.
We are therefore faced with a shape optimization problem. For a general overview on shape optimization we refer to \cite{delfour11}. 
To approach our particular problem numerically, we restrict ourselves to bounded domains that are \textit{strongly starshaped at 0}. Specifically, we assume that for any $D$ the boundary is given by
\[
	\partial D
	=\{r(q)q : q\in S^{d-1}\},
\]
for some suitably smooth radial function $r\colon S^{d-1}\to(0, \infty)$ on the $d$-sphere $S^{d-1}$.
Rather than optimizing over all such functions, we discretize the problem by considering $N\in\N$ points placed uniformly (in a suitable sense) 
on the sphere, say $\{q_i\}_{i=1}^N\subseteq S^{d-1}$, and then approximating any star-shaped set $D$ by the polytope $\widetilde{D}$ with vertices $\{p_i\}_{i=1}^N \coloneqq \{r(q_i)q_i\}_{i=1}^N$.
The polytope $\widetilde{D}$ has $N$ facets and hence can naturally be split into $N$ simplices, each with a vertex at the origin and the remaining vertices given by $p_{i_1},p_{i_2},\ldots,p_{i_d}$ for some appropriate set of indices $i_1, i_2, \ldots, i_d$.
Denote for such a set of indices $I=\{i_j\}_{j=1}^d$ by $F_I$ the facet of the simplex opposite the origin, and by $S_I$ the simplex itself.
Letting $\mathcal I$ be the family of such sets of indices, we arrive at the following approximation of the objective function
\begin{equation}\label{eq:approx_obj}
	J(D)
	\approx J(\widetilde{D})
	=\frac{1}{\sum_{I\in \mathcal I}\int_{S_I}\e^{-V(x)}\mathrm{d}x}\sum_{I\in \mathcal I}\Big(\int_{S_I}f(x)\e^{-V(x)}\,\mathrm{d}x+\kappa\int_{F_I}\e^{-V(x)}\, \mathcal H^{d-1}(\mathrm{d}x)\Big).
\end{equation}
Note that this approximation depends on $r$ (and hence on $D$) only through the $N$ lengths $\{r_i\}_{i=1}^N:=\{r(q_i)\}_{i=1}^N$, and hence we may consider $J$ as a function of $N$ variables, $J(\bm{r})\coloneqq J(\widetilde{D})$. 
Gradient based optimization schemes can now be used, since
\begin{equation}\label{eq:approx_gradient}
	\frac{\partial J(\bm{r})}{\partial r_i}
	=\frac{1}{C}\sum_{I\in \mathcal I_i}\Big(\frac{\partial}{\partial r_i}\Big(\int_{S_I}f(x)\e^{-V(x)}\,\mathrm{d}x+\kappa\int_{F_I}\e^{-V(x)}\,\mathcal H^{d-1}(\mathrm{d}x)\Big)
	-J(\bm{r})\frac{\partial}{\partial r_i}\Big(\int_{S_I}\e^{-V(x)}\,\mathrm{d}x\Big)\Big),
\end{equation}
where $C=\sum_{I\in \mathcal I}\int_{S_I}\e^{-V(x)}\, \mathrm{d}x$, and $\mathcal I_i=\{I\in \mathcal I: i\in I\}$.
To evaluate these expressions, we use the following theorem.

\begin{theorem}\label{theo:numeric}
	Let $q_1,\ldots,q_d\in S^{d-1}$ and $r_1,\ldots,r_d>0$ be given and let $p_i=r_iq_i$ for $i=1,\ldots,d$. Denote by $S$ the simplex in $\R^d$ spanned by the origin and the points $p_1,\ldots,p_d$ and by $F$ the facet of $S$ opposite the origin.
	Finally, let $P$ denote the $d\times d$ matrix whose $i$'th column is $p_i$ and $P_{-1}$ the $d\times(d-1)$ matrix whose $i$'th column is $p_{i+1}-p_1$.
	Then we have the following for $g\in \mathcal{C}(\R^d, \R)$:
	\begin{align}
		\int_{S}g(x)\,\mathrm{d}x
		&=|P|\int_{0}^{1}\int_{(0, 1)^{d-1}}g\big(r\eta(\bm{t})\big)\psi(\bm{t})r^{d-1}\mathrm{d}\bm{t}\ \mathrm{d}r\label{intS} \\
		\int_Fg(x)\,\mathcal H^{d-1}(\mathrm{d}x)
		&=\sqrt{|P_{-1}^TP_{-1}|}\int_{(0, 1)^{d-1}}g\big(\eta(\bm{t})\big)\psi(\bm{t})\,\mathrm{d}\bm{t}\label{intF}
	\end{align}
	where $\psi(t_1, \ldots, t_{d-1})=\prod_{i=1}^{d-2}t_i^{d-1-i}$, and
	\[
		\eta(t_1,\ldots,t_{d-1})
		=(1-t_1)p_1+t_1(1-t_2)p_2+\cdots+\Big(\prod_{i=1}^{d-2}t_i\Big)(1-t_{d-1})p_{d-1}+\Big(\prod_{i=1}^{d-1}t_i\Big)p_d.
	\]
	Furthermore, for $i=1,\ldots,d$, we have
	\begin{equation}
		\frac{\partial}{\partial r_i}
		\int_{S}g(x)\mathrm{d}x
		=\frac{1}{r_i}|P|\int_{(0, 1)^{d-1}}g\big(\eta_i(\bm{t})\big)\widehat{\psi}(\bm{t})\,\mathrm{d}\bm{t}\label{DintS},
	\end{equation}
	where $\widehat{\psi}(t_1,\ldots,t_{d-1})=(1-t_1)\psi(t_1,\ldots,t_{d-1})$, and $\eta_i$ denotes $\eta$ after swapping $p_1$ and $p_i$.
	Finally, if also $g\in \mathcal{C}^1(\R^d, \R)$,
	\begin{align}
		\frac{\partial}{\partial r_i}\int_{F}g(x)\, \mathcal H^{d-1}(\mathrm{d}x)
		&=\frac{1}{2}\tr\Big((P_{-1}^TP_{-1})^{-1}\frac{\partial}{\partial r_i}(P_{-1}^TP_{-1})\Big)\int_{(0, 1)^{d-1}}g\big(\eta(\bm{t})\big)\psi(\bm{t})\,\mathrm{d}\bm{t}\nonumber \\
		&\quad+\sqrt{\lvert P_{-1}^TP_{-1}\rvert}\int_{(0, 1)^{d-1}}\big\langle\nabla g\big(\eta_i(\bm{t})\big), q_i\big\rangle\widehat{\psi}(\bm{t})\diff{\bm{t}}.\label{DintF}
	\end{align}
\end{theorem}

\begin{proof}
	We first note that $\eta$ is simply repeated linear interpolation. That is, if $L(t; x, y)\coloneqq x+t(y-x)$ for $t\in(0, 1)$ and $x, y\in\R^d$, we have
	\[
		\eta(t_1, \ldots, t_{d-1})
		=L\Big(t_1; p_1, L\big(t_2; p_2, L(t_3; p_3\ldots)\big)\Big).
	\]
	As such, for any point $x\in\interior F$, there exists a unique $\bm{t}\in(0, 1)^d$ such that $\eta(\bm{t})=x$, and similarly for any $y\in\interior S$, a unique $r\in(0, 1)$ such that $r\eta(\bm{t})=y$.
	As such, to verify \eqref{intS} and \eqref{intF}, we only need to show that the functions
	\[
		(r, \bm{t})\mapsto|P|\psi(\bm{t})r^{d-1},\qquad\text{and}\qquad
		\bm{t}\mapsto\sqrt{|P_{-1}^TP_{-1}|}\psi(\bm{t}),
	\]
	denote the Jacobian and Gramian of $(r, \bm{t})\mapsto r\eta(\bm{t})$ and $\eta$, respectively.
	To this end, we introduce the following notation: let $t_d\equiv0$, $T_1=1$ and set $T_i=\prod_{j=1}^{i-1}t_j$ for $i=2,\ldots,d$.
	Then we may write $\eta(t_1,\ldots,t_{d-1})=\sum_{i=1}^{d}T_i(1-t_i)p_i$, whereby
	\[
		\frac{\partial}{\partial t_i}r\eta(t_1,\ldots,t_{d-1})
		=r\Big(\frac{\eta(t_1,\ldots,t_{d-1})-\sum_{j=1}^{i-1}T_j(1-t_j)p_j-T_ip_i}{t_i}\Big),\quad i=1,\ldots,d-1,
	\]
	while of course $\frac{\partial}{\partial r}r\eta=\eta$.
	From this it follows by determinant properties that the Jacobian of $(r, \bm{t})\mapsto r\eta(\bm{t})$ is given by
	\[
		|\mathcal J(r, \bm{t})|
		=r^{d-1}\Big(\prod_{i=1}^{d-1}\frac{1}{t_i}\Big)\Big(\prod_{i=1}^{d}T_i\Big)|P|
		=|P|\psi(\bm{t})r^{d-1},
	\]
	showing \eqref{intS} as desired.

	To show \eqref{intF}, we note that since $F$ lies in a $d-1$-dimensional hyperplane, say $H$, we may embed it in $\R^{d-1}$ by an isometry $\Psi\colon H\to\R^{d-1}$ with $\Psi(p_1)=0$.
	Specifically, $\Psi$ can be constructed as $x\mapsto\Psi'\big(A(x-p_1)\big)$, where $A$ is the rotation matrix such that the $d$'th coordinate of $Ax$ is $0$ for all $x\in H$, and $\Psi'\colon\R^d\to\R^{d-1}$ simply discards the last coordinate.
	Then, integrating a function $g$ over $F$ with respect to $\mathcal H^{d-1}$ is equivalent to integrating $g\circ\Psi^{-1}$ over $\Psi(F)$ with respect to the $d-1$-dimensional Lebesgue measure.
	Now, since $\Psi(F)$ by the above construction of $\Psi$ is a simplex in $\R^{d-1}$ consisting of the origin and $d-1$ other points, say $p_1',\ldots,p_{d-1}'$, it follows by the above,
	\[
		\int_{F'}g\circ\Psi^{-1}(x)\mathrm{d}x
		=|P'|\int_{0}^{1}\int_{(0, 1)^{d-2}}g\circ\Psi^{-1}\big(r\eta'(\bm{t})\big)\psi'(\bm{t})r^{d-1}\diff{\bm{t}}\ \mathrm{d}r,
	\]
	where $\eta'$ similarly is linear interpolation between $p_1',\ldots,p_{d-1}'$, $\psi'$ is the $d-1$-dimensional equivalent of $\psi$ and $P'$ is the $(d-1)\times(d-1)$ matrix whose $i$'th column is $p_i'$.
	By some elementary substitutions and renaming of variables, we may write in an abuse of notation $\Psi^{-1}\big(r\eta'(\bm{t})\big)=\eta(\bm{t})$ and $\psi'(\bm{t})r^{d-1}=\psi(\bm{t})$.
	Finally, to find $|P'|$, we see
	\[
		|P'|
		=(d-1)!\Vol_{d-1} F'
		=(d-1)!\Vol_{d-1} F
		=\sqrt{|P_{-1}^TP_{-1}|},
	\]
	where we use that $\Psi$ is an isometry and hence preserves volumes.

	To show \eqref{DintS}, let $i\in\{1,\ldots,d\}$ be fixed and consider now for some small $h$ the simplex $S'$ with a vertex at the origin and at the points $p_1,\ldots,p_{i-1},\frac{r_i+h}{r_i}p_i,p_{i+1},\ldots,p_d$.
	Since either $S\subseteq S'$ or $S'\subseteq S$, the symmetric difference $S\triangle S'$ is another simplex with vertices at $p_1,\ldots,p_d$ and $\frac{r_i+h}{r_i}p_i$.
	Assume without loss of generality that $S\subseteq S'$.
	Shifting the coordinate system so that $p_i$ lies at the origin, we get a simplex with a vertex at the origin and at the points $p_1-p_i,p_2-p_i,\ldots,\frac{h}{r_i}p_i,\ldots,p_d-p_i$.
	Note that by properties of the determinant, we get
	\[
		\Big\lvert\mat{p_1-p_i & p_2-p_i & \cdots & \frac{h}{r_i}p_i & \cdots & p_d-p_i}\Big\rvert
		=\frac{h}{r_i}\lvert P \rvert.
	\]
	Using this and \eqref{intS}, we find that
	\begin{align*}
		\int_{S'}g(x)\mathrm{d}x-\int_Sg(x)\mathrm{d}x
		=\int_{S'\setminus S}g(x)\mathrm{d}x
		=\frac{h}{r_i}|P|\int_{0}^{1}\int_{(0, 1)^{d-1}}g\big(r\eta_{i,h}(\bm{t})+p_i\big)r^{d-1}\psi(\bm{t})\diff{\bm{t}}\ \mathrm{d}r,
	\end{align*}
	where, using the same notation as earlier,
	\[
		\eta_{i, h}(t_1, \ldots, t_{d-1})
		=T_i(1-t_i)\frac{h}{r_i}p_i+\sum_{j\neq i}T_j(1-t_j)(p_j-p_i).
	\]
	Dividing by $h$ and letting $h\to0$ (implicitly using dominated convergence and the continuity of $g$), we thus find
	\begin{equation}
		\frac{\partial}{\partial r_i}
		\int_{S}g(x)\mathrm{d}x
		=\frac{1}{r_i}|P|\int_{0}^{1}\int_{(0, 1)^{d-1}}g\big(r\eta_{i, 0}(\bm{t})+p_i\big)r^{d-1}\psi(\bm{t})\diff{\bm{t}}\diff{r}.\label{DintS_bad}
	\end{equation}
	Noting that
	\[
		\sum_{j\neq i}T_j(1-t_j)
		=\sum_{j\neq i}^{d-1}(T_j-T_{j+1})+T_d
		=T_1-(T_i-T_{i+1})
		=1-T_i(1-t_i),
	\]
	we find
	\begin{align*}
		r\eta_{i, 0}(t_1,\ldots,t_{d-1})+p_i
		&=r\Big(\sum_{j\neq i}T_j(1-t_j)(p_j-p_i)\Big)+p_i \\
		&=r\Big(\sum_{j\neq i}T_j(1-t_j)p_j-\big(1-T_i(1-t_i)\big)p_i\Big)+p_i \\
		&=r\eta(t_1,\ldots,t_{d-1})+(1-r)p_i,
	\end{align*}
	which together with \eqref{DintS_bad} yields
	\[
		\frac{\partial}{\partial r_i}
		\int_{S}g(x)\mathrm{d}x
		=\frac{1}{r_i}|P|\int_{0}^{1}\int_{(0, 1)^{d-1}}g\big(r\eta(\bm{t})+(1-r)p_i\big)r^{d-1}\psi(\bm{t})\diff{\bm{t}}\diff{r}.
	\]
	At this point we remark that the ordering of the points $p_1,\ldots,p_d$ is arbitrary and has no influence on the value of the above integrals.
	As such, we may swap the places of two points, say $p_1$ and $p_i$, and thus replacing $\eta$ by $\eta_i$ in the above integral.
	From this, we see
	\[
		r\eta_i(t_1,\ldots,t_{d-1})+(1-r)p_i
		=(1-rt_1)p_i+rt_1(1-t_2)p_2+\cdots+rt_1\Big(\prod_{j=2}^{d-1}t_j\Big)p_d
		=\eta_i(rt_1,\ldots,t_{d-1}).
	\]
	Thus, making the substitution $u=rt_1$ in the above integral, we get
	\begin{align*}
		\int_{0}^{1}\int_{(0, 1)^{d-1}}g\big(r\eta(\bm{t})+(1-r)p_i\big)r^{d-1}\psi(\bm{t})\diff{\bm{t}}\diff{r}
		&=\int_{0}^{1}\int_{(0, 1)^{d-2}}\int_{0}^{r}g\big(\eta_i(u, \bm{t})\big)\psi(u, \bm{t})\mathrm{d}u \diff{\bm{t}} \diff{r} \\
		&=\int_{(0, 1)^{d-2}}\int_{0}^{1}g\big(\eta(u, \bm{t})\big)\psi(u, \bm{t})(1-u)\diff{u} \diff{\bm{t}} \\
		&=\int_{(0, 1)^{d-1}}g\big(\eta_i(\bm{t})\big)\widehat{\psi}(\bm{t})\diff{\bm{t}},
	\end{align*}
	which shows \eqref{DintS}.
	Finally, to show \eqref{DintF}, we have by Jacobi's formula, 
	\[
		\frac{\partial}{\partial r_i}|P_{-1}^TP_{-1}|
		=|P_{-1}^TP_{-1}|\ \tr\Big((P_{-1}^TP_{-1})^{-1}\frac{\partial}{\partial r_i}(P_{-1}^TP_{-1})\Big),
	\]
	whereby
	\begin{align*}
		\frac{\partial}{\partial r_i}\int_{F}g(x)\, \mathcal H^{d-1}(\mathrm{d}x)
		&=\frac{1}{2}\tr\Big((P_{-1}^TP_{-1})^{-1}\frac{\partial}{\partial r_i}(P_{-1}^TP_{-1})\Big)\int_{(0, 1)^{d-1}}g\big(\eta(\bm{t})\big)\psi(\bm{t})\,\mathrm{d}\bm{t} \\
		&\qquad+\sqrt{\lvert P_{-1}^TP_{-1}\rvert}\frac{\partial}{\partial r_i}\Big(\int_{(0, 1)^{d-1}}g\big(\eta_i(\bm{t})\big)\psi(\bm{t})\diff{\bm{t}}\Big) \\
		&=\frac{1}{2}\tr\Big((P_{-1}^TP_{-1})^{-1}\frac{\partial}{\partial r_i}(P_{-1}^TP_{-1})\Big)\int_{(0, 1)^{d-1}}g\big(\eta(\bm{t})\big)\psi(\bm{t})\,\mathrm{d}\bm{t}\nonumber \\
		&\qquad+\sqrt{|P_{-1}^TP_{-1}|}\int_{(0, 1)^{d-1}}\big\langle\nabla g\big(\eta_i(\bm{t})\big), q_i\big\rangle\widehat{\psi}(\bm{t})\diff{\bm{t}},
	\end{align*}
	where we in the last equality used $\frac{\partial}{\partial r_i}\eta_i(t_1, \ldots, t_{d-1})=(1-t_1)q_i$.
\end{proof}

When the dynamics of the process are known, that is we have access to the potential $V$, the theorem allows us to find explicit expressions for the right hand side of \eqref{eq:approx_obj} and \eqref{eq:approx_gradient} given  $N$ points on the sphere and therefore to numerically optimize within the space of polytope approximations to star-shaped sets via gradient descent. For the two-dimensional case, strikingly simple  expressions can be derived from this and, given sufficient regularity of the star-shaped domain, the approximation rate in \eqref{eq:approx_obj} is at most of order $1/N$.

\begin{corollary}\label{coro:twodim}
Let $r\in \mathcal{C}^2([0, 2\pi], [r_{\min}, r_{\max}])$ with $0<r_{\min}\le r_{\max}<\infty$ be some periodic radial function, and denote by $D\subseteq\R^2$ the set with $\partial D=r([0, 2\pi])$.
For any $N\in\N$, let $r_i=r(\frac{2 i\pi}{N})$, $q_i=(\cos\frac{2i\pi}{N},\sin\frac{2i\pi}{N})$ and $p_i=r_iq_i$ and denote by $\widetilde{D}\subseteq\R^2$ the simplex with vertices $\{p_i\}_{i=1}^N$.
	Finally, for $i=1,\ldots,N$ and $t\in(0,1)$, let $\eta_i^+(t)=p_i+t(p_{i+1}-p_i)$ and $\eta_i^-(t)=p_i+t(p_{i-1}-p_i)$, where we identify $p_{0}=p_N$ and $p_{N+1}=p_1$.
	Then, there exists a constant $K\ge0$ such that
	\[
		|J(D)-J(\widetilde{D})|
		\le\frac{K}{N},
	\]
 where for $\tilde{\rho} = \e^{-V}$ we have the explicit representations
	\[
		J(\widetilde{D})
		=\frac{1}{C}\sum_{i=1}^{N}\Big(
			\sin\frac{2\pi}{N}r_ir_{i+1}\int_{0}^{1}\int_{0}^{1}(f\tilde{\rho})\big(r\eta_i^+(t)\big)r \diff{r} \diff{t}
			+\kappa\n{p_{i+1}-p_{i}}\int_{0}^{1}\tilde{\rho}\big(\eta_i^+(t)\big)\diff{t}
		\Big),
	\]
	and
	\begin{align*}
		\frac{\partial J(\widetilde{D})}{\partial r_i}
		=\frac{1}{C} \int_0^1 \big(\psi^+(t)\tilde\rho\big(\eta_i^+(t)\big) + \psi^-(t)\tilde\rho\big(\eta_i^-(t)\big)\big) \diff{t}, 
    \end{align*}
	where 
    \[\psi^\pm(t) = \big(\sin\frac{2\pi}{N}r_{i\pm1}\big(f\big(\eta_i^\pm(t)\big)-J(\widetilde{D})\big)-\kappa\n{p_{i\pm1}-p_i}\big\langle \nabla V\big(\eta_i^\pm(t)\big), q_i\big\rangle
					\big)(1-t)
					+\frac{\kappa(r_i-\cos\frac{2\pi}{N}r_{i\pm1})}{\n{p_i-p_{i\pm1}}},\]
    and 
	\[
		C
		=\sum_{i=1}^{N}
			\sin\frac{2\pi}{N}r_ir_{i+1}\int_{0}^{1}\int_{0}^{1}\tilde\rho\big(r\eta_i^+(t)\big)r \diff{r} \diff{t}.
	\]
\end{corollary}

\begin{proof}
	We show only the first claim, the rest is an immediate consequence of Theorem \ref{theo:numeric}.
	To this end, we assume for convenience that $N\ge4$ and introduce the following notation: for $\theta\in[0, 2\pi]$ and any radial function $\widehat{r}\colon [0, 2\pi]\to(0, \infty)$ which is $\mathcal{C}^2$ almost everywhere, let $p(\theta)=(\cos\theta, \sin\theta)$ and $\widehat{\bm{r}}(\theta)=(\widehat{r}(\theta), \widehat{r}'(\theta))$ when the derivative exists.
	Also, for any continuous function $g\colon \R^2\to\R$, write
	\[
		I(\widehat{r}, g)
		\coloneqq\int_{0}^{2\pi}\int_{0}^{\widehat{r}(\theta)}g\big(sp(\theta)\big)s\diff{s} \diff{\theta},\quad\text{and}\quad
		S(\widehat{r}, g)
		\coloneqq\int_{0}^{2\pi}g\big(\widehat{r}(\theta)p(\theta)\big)\n{\widehat{\bm{r}}(\theta)}\diff{\theta}.
	\]
	It is then straightforward to check that
	\[
		J(D)
		=\frac{I(r, f\tilde\rho)+\kappa S(r, \tilde\rho)}{I(r, \tilde\rho)}.
	\]
	Letting $\widetilde{r}\colon[0, 2\pi]\to[r_{\min}, r_{\max}]$ denote the radial function corresponding to $\widetilde{D}$, we then have
	\[
		|J(D)-J(\widetilde{D})|
		=\frac{1}{I(r, \tilde\rho)}\Big|(I(r, f\tilde\rho)-I(\widetilde{r}, f\tilde\rho))+\kappa(S(r, \tilde\rho)-S(\widetilde{r}, \tilde\rho))+J(\widetilde{D})(I(r, \tilde\rho)-I(\widetilde{r}, \tilde\rho))\Big|.
	\]
	By our assumptions on $r$, it follows that $I(r, \tilde\rho)$ is bounded from below and $J(\widetilde{D})$ from above.
	As such, to show the desired convergence rate, we only need to show that for any continuous $g\colon\R^2\to\R$,
	\begin{equation}\label{cor:conv_rate_1}
		|I(r, g)-I(\widetilde{r}, g)|\vee|S(r, \tilde\rho)-S(\widetilde{r}, \tilde\rho)|
		\le\frac{K_1}{N}		
	\end{equation}
	for some $K_1\ge0$.
	In fact, by dominating with $\sup_{x\in B(0, r_{\max})}|g(x)|$, it suffices to show \eqref{cor:conv_rate_1} with $g\equiv1$.
	This follows if we show for $i=1,\ldots,N$
	\begin{equation}\label{cor:conv_rate_2}
		|r(\theta)-\widetilde{r}(\theta)|\vee|r'(\theta)-\widetilde{r}'(\theta)|
		\le\frac{K_2}{N},\quad\theta\in\big(\tfrac{2(i-1)\pi}{N}, \tfrac{2i\pi}{N}\big)
	\end{equation}
	for some $K_2\ge0$ independent of $i$.
	Indeed, if this is the case, we have first,
	\[
		|I(r, 1)-I(\widetilde{r}, 1)|
		\le\sum_{i=1}^{N}\int_{\frac{2(i-1)\pi}{N}}^{\frac{2i\pi}{N}}\Big|\int_{\widetilde{r}(\theta)}^{r(\theta)}s\ \mathrm{d}s\Big|\ \mathrm{d}\theta
		\le r_{\max}\sum_{i=1}^{N}\int_{\frac{2(i-1)\pi}{N}}^{\frac{2i\pi}{N}}|r(\theta)-\widetilde{r}(\theta)|\ \mathrm{d}\theta
		\le \frac{2\pi r_{\max}K_2}{N},
	\]
	and similarly, since $\n{\bm{r}-\widetilde{\bm{r}}}=\sqrt{(r-\widetilde{r})^2+(r'-\widetilde{r}')^2}\le\sqrt{2}\frac{K_2}{N}$,
	\begin{align*}
		\lvert S(r, \tilde\rho)-S(\widetilde{r}, \tilde\rho) \rvert
		&\le\sum_{i=1}^N\int_{\frac{2(i-1)\pi}{N}}^{\frac{2i\pi}{N}}\big(\n{\bm{r}(\theta)}\big\lvert\tilde\rho\big(r(\theta)p(\theta)\big)-\tilde\rho\big(\widetilde{r}(\theta)p(\theta)\big)\big\rvert +\tilde\rho\big(\widetilde{r}(\theta)p(\theta)\big)\big\lvert \n{\bm{r}(\theta)}-\n{\widetilde{\bm{r}} (\theta)}\big\rvert\big)\mathrm{d}\theta \\
		&\le K_3\sum_{i=1}^{N}\int_{\frac{2(i-1)\pi}{N}}^{\frac{2i\pi}{N}} \big(\lvert r(\theta)-\widetilde{r}(\theta)\rvert +\n{\bm{r}(\theta)-\widetilde{\bm{r}}(\theta)} \big)\diff{\theta}
		\le\frac{2\pi(1+\sqrt{2})K_2K_3}{N},
	\end{align*}
	where
	\[
		K_3
		=\Big(\big(\sup_{\theta\in[0, 2\pi]}\n{\bm{r}(\theta)}\sup_{x\in \operatorname{cl}B(0, r_{\max})}\n{\nabla\tilde\rho(x)}\big)\vee\sup_{x\in \operatorname{cl} B(0, r_{\max})}\tilde\rho(x)\Big)
	\]
	whereby we may choose $K_1=2\pi K_2(r_{\max}\vee(1+\sqrt{2})K_3)$.
	To verify \eqref{cor:conv_rate_2}, we first find by some straightforward trigonometry, identifying $r_0=r_N$,
	\[
		\widetilde{r}(\theta)
		=\sum_{i=1}^{N}\frac{r_i(m_i\cos\frac{2i\pi}{N}-\sin\frac{2i\pi}{N})}{m_i\cos\theta-\sin\theta}\bm{1}_{[\frac{2(i-1)\pi}{N},\frac{2i\pi}{N})}(\theta),\quad\text{where}\quad
		m_i
		=\frac{r_i\sin\frac{2i\pi}{N}-r_{i-1}\sin\frac{2(i-1)\pi}{N}}{r_i\cos\frac{2i\pi}{N}-r_{i-1}\cos\frac{2(i-1)\pi}{N}}.
	\]
 Here, when $r_i\cos\frac{2i\pi}{N}=r_{i-1}\cos\frac{2(i-1)\pi}{N}$ we set $m_i=\infty$ and treat $\tilde{r}$ in a limiting sense. 
	Now, for $i\in\{1, \ldots, N\}$ and $\theta\in(\frac{2(i-1)\pi}{N}, \frac{2i\pi}{N})$, let $\varphi_i(\theta)$ denote the $i$'th summand in $\widetilde{r}(\theta)$.
	Let also $\psi_i(\theta)=\frac{m_i\sin\theta+\cos\theta}{m_i\cos\theta-\sin\theta}$, such that $\varphi_i'(\theta)=\varphi_i(\theta)\psi_i(\theta)$.
	Further, since $\psi_i'(\theta)=1+\psi_i(\theta)^2$, we find also $\varphi_i''(\theta)=\varphi_i(\theta)(1+2\psi_i(\theta)^2)$, whereby to bound $\widetilde{r}$, $\widetilde{r}'$ and $\widetilde{r}''$ (when the latter two exist), we need only bound $\varphi_i$ and $\psi_i$ for all $i$.
	Clearly, we have $\varphi_i(\theta)\le r_{\max}$, and since $\psi_i'(\theta)>0$, we must have $|\psi_i(\theta)|\le|\psi_i(\frac{2(i-1)\pi}{N})|\vee|\psi_i(\frac{2i\pi}{N})|$.
	Some simple, albeit tedious, applications of trigonometric identities show that
	\[
		\psi_i\Big(\frac{2(i-1)\pi}{N}\Big)
		=-\frac{r_{i-1}-r_i\cos\frac{2\pi}{N}}{r_i\sin\frac{2\pi}{N}}
		\qquad\text{and}\qquad
		\psi_i\Big(\frac{2i\pi}{N}\Big)
		=\frac{r_i-r_{i-1}\cos\frac{2\pi}{N}}{r_{i-1}\sin\frac{2\pi}{N}},
	\]
	and since e.g. $\lvert r_i-r_{i-1}\cos\frac{2\pi}{N}\rvert\le\cos\frac{2\pi}{N}\lvert r_i-r_{i-1}\rvert +r_i(1-\cos\frac{2\pi}{N})$, we find by the Lipschitz-continuity of $r$ and  cosine,
	\[
		\lvert\psi_i(\theta)\rvert
		\le\frac{2\pi}{Nr_{\min}}\Big(\frac{1}{\tan\frac{2\pi}{N}}\sup_{\widehat{\theta}\in[0, 2\pi)}\lvert r'(\widehat{\theta})\rvert+r_{\max}\Big)
		\lesssim\frac{\sup_{\widehat{\theta}\in[0, 2\pi]}\lvert r'(\widehat{\theta})\rvert}{r_{\min}}.
	\]
        Here we use that $\frac{x}{\tan x}\in[0, 1]$ for $0\le x\le\frac{\pi}{2}$, and hence $\frac{2\pi}{N\tan\frac{2\pi}{N}}\in[0, 1]$ for $N\ge4$.
	Since this bound depends neither on $i$ nor $\theta$, we find for all $\theta\in[0, 2\pi)\setminus\{\frac{2i\pi}{N}\}_{i=1}^N$
	\[
		\widetilde{r}(\theta)
		\le r_{\max},\quad
		\lvert \widetilde{r}'(\theta)\rvert
		\lesssim\frac{r_{\max}}{r_{\min}}\sup_{\widehat{\theta}\in[0, 2\pi]}|r'(\widehat{\theta})|,\quad\text{and}\quad
		|\widetilde{r}''(\theta)|
		\lesssim r_{\max}\Big(1+2\Big[\frac{1}{r_{\min}}\sup_{\widehat{\theta}\in[0, 2\pi]}|r'(\widehat{\theta})|\Big]^2\Big).
	\]
	Thus we have by Lipschitz-continuity for $i\in\{1, \ldots, N\}$ and $\theta\in[\frac{2(i-1)\pi}{N}, \frac{2i\pi}{N})$, since $r(\frac{2(i-1)\pi}{2})=r_{i-1}=\widetilde{r}(\frac{2(i-1)\pi}{N})$
	\[
		\lvert r(\theta)-\widetilde{r}(\theta)\rvert
		\le\lvert r(\theta)-r_{i-1}|+|r_{i-1}-\widetilde{r}(\theta)\rvert
		\le\frac{2\pi(1+\frac{r_{\max}}{r_{\min}})\sup_{\widehat{\theta}\in[0, 2\pi]}\lvert r'(\widehat{\theta})\rvert}{N}
	\]
	Furthermore, since $\widetilde{r}$ and $r$ agree on $\{\frac{2i\pi}{N}\}_{i=1}^N$, it follows by the mean value theorem that there exist $\theta^*,\widetilde{\theta}^*\in(\frac{2(i-1)\pi}{N},\frac{2i\pi}{N})$, such that $r'(\theta^*)=\widetilde{r}'(\widetilde{\theta}^*)$.
	Thus we also have
	\begin{align*}
		\lvert r'(\theta)-\widetilde{r}'(\theta)\rvert 
		&\le\lvert r'(\theta)-r'(\theta^*)|+|\widetilde{r}'(\widetilde{\theta}^*)-\widetilde{r}'(\theta) \rvert\\
		&\le\frac{2\pi(\sup_{\widehat{\theta}\in[0, 2\pi]}\lvert r''(\widehat{\theta})\rvert+r_{\max}(1+2(\frac{1}{r_{\min}}\sup_{\widehat{\theta}\in[0, 2\pi]}\lvert r'(\widehat{\theta})\rvert)^2))}{N},
	\end{align*}
	and the proof is finished by setting
	\[
		K_2
		=2\pi\Big(\big((1+\tfrac{r_{\max}}{r_{\min}})\sup_{\widehat{\theta}\in[0, 2\pi]}\lvert r'(\widehat{\theta})\rvert\big)\vee\big(\sup_{\widehat{\theta}\in[0, 2\pi]}\lvert r''(\widehat{\theta})\rvert+r_{\max}(1+2(\tfrac{1}{r_{\min}}\sup_{\widehat{\theta}\in[0, 2\pi]}\lvert r'(\widehat{\theta})\rvert)^2)\big)\Big).
	\]
\end{proof}

This now implies for $d=2$ that for suitable admissible domain families $\Theta \ni D$, the infimum of $D \mapsto J(D)$ over $\Theta$ is well approximated by the infimum over the polytope approximations. 
\begin{corollary}
Let $d=2$ and let $\Theta$ be the family of domains $D$ that are strongly starhaped at $0$ and are identified by $\mathcal{C}^2$ periodic radial functions $r_D \colon [0,2\pi] \to (0,\infty)$ such that for some global constants $\underline{\lambda},\overline{\lambda},\Lambda$,
\[\underline{\lambda} \leq r_D \leq \overline{\lambda} \quad \text{and} \quad \max_{\theta \in [0,2\pi]} (\lvert r^\prime_D(\theta) \rvert + \lvert r^{\prime \prime}_D(\theta) \rvert) \leq \Lambda.\] 
Then, letting $\tilde{D}_N$ be the polytope approximation of $D$ from Corollary \ref{coro:twodim}, it holds 
\[\big\lvert \inf_{D \in \Theta} J(\tilde{D}_N) - \inf_{D \in \Theta} J(D) \big\rvert \lesssim 1/N.\]
\end{corollary}
\begin{proof} 
For given $\varepsilon > 0$ choose $D^\varepsilon \in \Theta$ s.t.\ $J(D^\varepsilon) \leq \inf_{D \in \Theta} J(D) + \varepsilon$. Then, by Corollary \ref{coro:twodim},
\[\inf_{D \in \Theta} J(\tilde{D}_N) - \inf_{D \in \Theta} J(D) \leq \inf_{D \in \Theta} J(\tilde{D}_N) - J(D^\varepsilon) + \varepsilon \leq J(\tilde{D^\varepsilon}_N) - J(D^\varepsilon) + \varepsilon \leq \sup_{D \in \Theta} (J(\tilde{D}_N) - J(D)) + \varepsilon \lesssim \frac{1}{N} + \varepsilon.\]
Similarly, letting $D^\varepsilon \in \Theta$ such that $J(\tilde{D^\varepsilon}_N) \leq \inf_{D \in \Theta} J(\tilde{D}_N) + \varepsilon$, we have,
\[\inf_{D \in \Theta} J(D) - \inf_{D \in \Theta} J(\tilde{D}_N) \leq   J(D^\varepsilon)- J(\tilde{D^\varepsilon}_N) + \varepsilon \leq \sup_{D \in \Theta} \lvert J(\tilde{D}_N) - J(D) \rvert + \varepsilon \lesssim \frac{1}{N} + \varepsilon.\] 
Taking together both bounds and letting $\varepsilon \downarrow 0$ therefore gives the result.
\end{proof}

In the  simple case of a pure Brownian motion, corresponding to $\e^{-V}\equiv1$, and given radially symmetric costs $f=\n{\cdot}$, one expects the optimal reflection boundary to be a sphere centered at $0$. Optimizing the corresponding cost functional over the space of such balls only, gives the optimization problem a parametric structure that can be easily solved  analytically to reveal the optimal ball to be $D^*=B(0, r^\ast)$, where $r^\ast=\sqrt{(d+1)\kappa}$. It is now interesting to test our method with regard to two questions: does the numerical optimization over the more general class of star-shaped domains support our intuition by identifying a ball as the optimal reflection domain, and if so, do we obtain a good approximation of the optimal radius $r^\ast=\sqrt{(d+1)\kappa}$ as well? The result for different choices of $\kappa$  is visualized in  Figure \ref{fig:optimal_radii}, giving an affirmative answer to both questions.
\begin{figure}[h]
	\centering
	\includegraphics[width=0.6\textwidth]{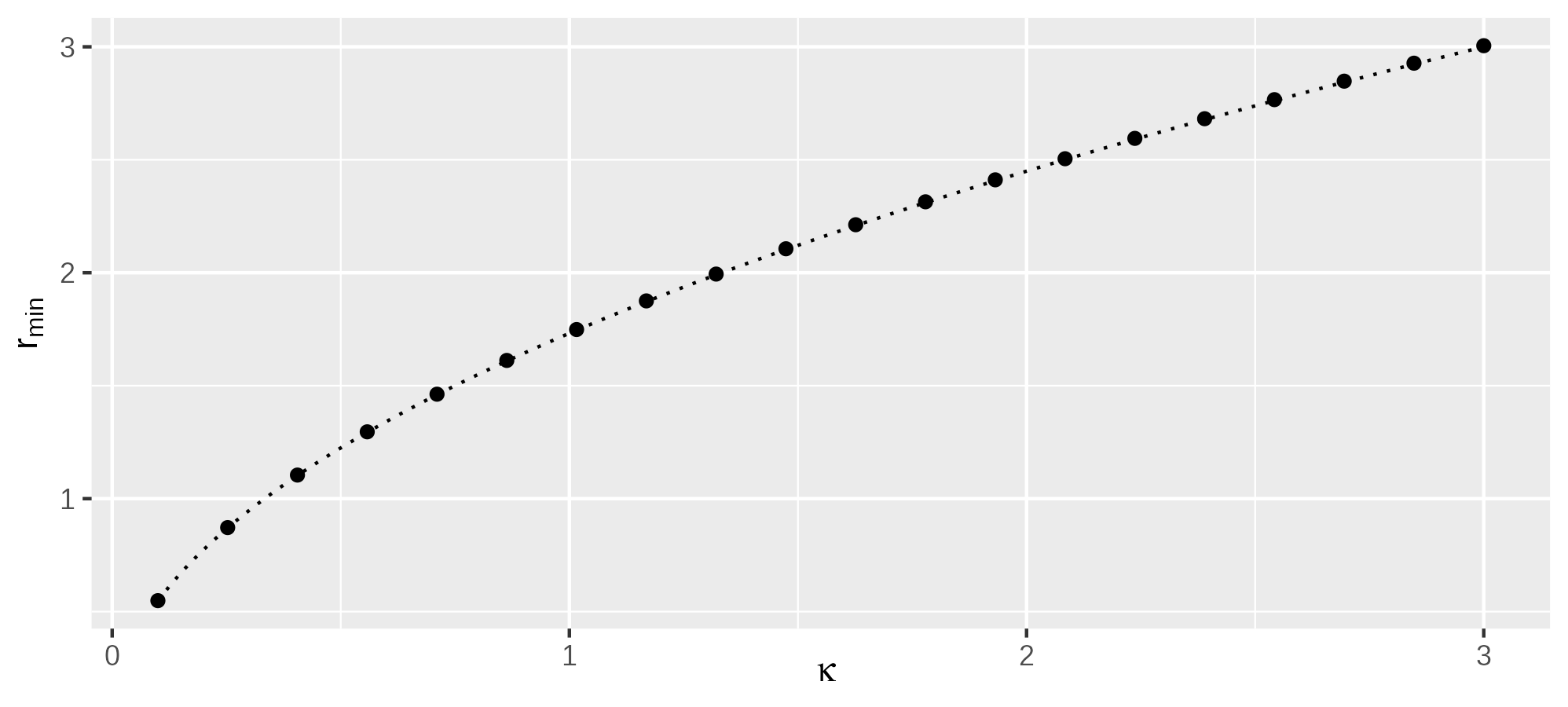}
	\caption{For each value of $\kappa$, we use a quasi-Newton method for finding the optimal shape according to the gradients above.
 Each time, a ball is indeed identified as the optimal shape and we plot its radius as a function of $\kappa$.
 Finally the dotted line represents the optimal line $r^\ast=\sqrt{3\kappa}$.}
 \label{fig:optimal_radii}
\end{figure}

Our method is also well-equipped for handling more challenging non-symmetric situations, where it is hard to make an educated guess on the optimal shape.
As such, we also test the method on reflected Ornstein--Uhlenbeck processes with strong correlation, in particular the processes
\[
	\mathrm{d}X_t^D
	=AX_t^D \,\mathrm{d}t+\sqrt{2}\,\mathrm{d}W_t+n(X_t^D)\,\mathrm{d}L_t^D,\quad
	A
	=\mat{1 & 0.9 \\ 0.9 & 1}^{-1},
\]
as well as a skewed cost-function, namely $f(x, y)=\sqrt{x^2+5y^2}$.
The found approximately optimal shapes can be seen in Figure \ref{fig:simulated}. Here,  for each shape we take $N=50$, $\kappa=1$ and use a quasi-Newton method to find the optimal shape with starting values $r_i=1$ for $i=1,\ldots,50$.
The found optimal shapes correspond to, from top-left to bottom-right, Brownian motion with norm cost, Ornstein--Uhlenbeck process with norm cost, Brownian motion with skewed cost and Ornstein--Uhlenbeck process with skewed cost.
\begin{figure}[ht]
	\centering
	\includegraphics[width=0.6\textwidth]{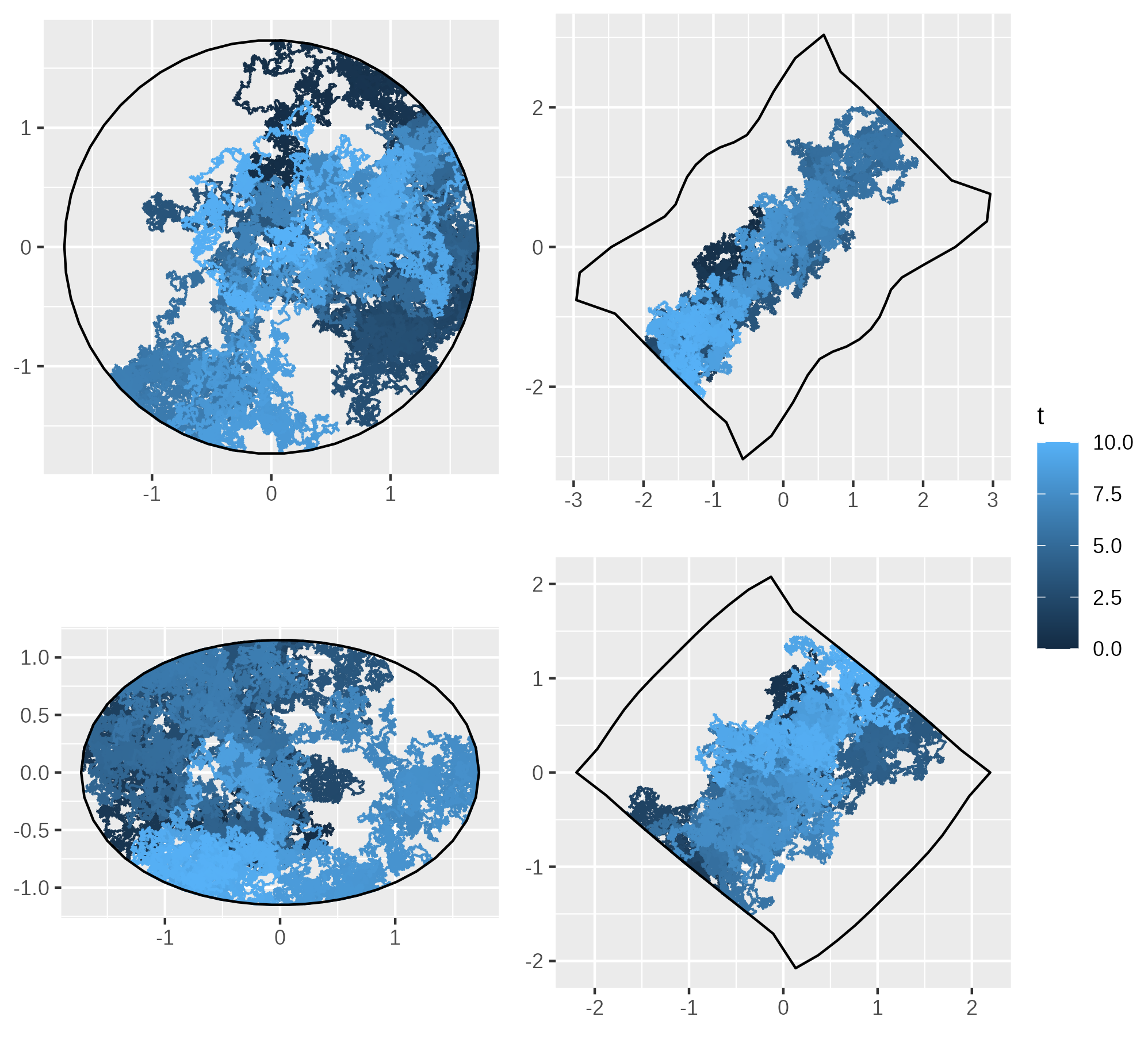}
	\caption{
    Simulated optimal shapes and corresponding path realizations of reflected processes.
	}
 \label{fig:simulated}
\end{figure}
Finally, Figure \ref{fig:simulated} also shows simulations of the above reflected processes in these approximately optimal shapes to asses the convergence of the realized costs towards the theoretical objective function.
To do this, we use the usual Euler scheme, where we then project the process onto $D$ whenever it leaves the domain.
The distance of this projection is then the associated simulated increase in local time, cf.\ \cite{slo94}. We simulate the relevant processes  with time-steps of $10^{-4}$ until time $T=100$ (but plot only until $T=10$ for visual clarity). The average realized cost in comparison to the expected average long run costs are given in Table \ref{tab:sim}.

\begin{table}[h]
\centering
\begin{tabular}{l | c | c }
& norm cost function & skewed cost function \\\hline
Brownian motion & $2.22$ ($2.31$) & $2.83$ ($2.91$) \\ \hline 
Ornstein--Uhlenbeck & $1.18$ ($1.15$) & $1.66$ ($1.74$) 
\end{tabular}
\caption{Average realized costs vs.\ expected average long term costs (in brackets)}
\label{tab:sim}
\end{table}

\section{Learning the optimal boundary}\label{sec:learning}
We now turn to the challenging situation, when the dynamics of the unconstrained Langevin diffusion are unknown, which makes it impossible to set the  optimization algorithm from the previous section into motion, without feeding it information based on collected data first. As apparent from the explicit form of the cost functional given in Corollary \ref{coro:J_stat}, a natural data-driven reflection procedure  can be based on a plug-in approach, provided that we have an efficient estimator of (functionals of) the invariant density of the unconstrained Langevin diffusion at our disposal. 

\subsection{Adaptive nonparametric estimation of the invariant density}
As in the scalar case discussed in \cite{christensen21}, we employ a kernel estimator of the invariant density, whose $\sup$-norm risk given appropriate conditions on the diffusion coefficients is well-understood in a general context by now. In the following, we will concentrate on a class of potentials $V$ s.t.\ the  process satisfies certain functional inequalities. This setting is quite natural given the reversible nature of Langevin diffusions and is studied in \cite{str18}, where minimax optimal estimation rates for a Lepski type \textit{adaptive} kernel estimator are established under \textit{anisotropic} Hölder smoothness assumptions on the invariant density. To recall these results, some preliminary definitions are necessary.

Let $\pi$ be the invariant distribution of $X$ with density $\rho \propto \exp(-V)$ and let $\mathrm{L}$ be the $L^2(\pi)$-generator of $X$ with domain $\mathcal{D}(\mathrm{L})$. Since $-\mathrm{L}$ is self-adjoint and nonnegative on the Hilbert space $L^2(\pi)$ endowed with the inner product $\langle f, g\rangle_{\pi} \coloneqq \int fg \diff{\pi}$, we may define $\sqrt{-\mathrm{L}}$ via spectral calculus and note that for any $f \in \mathcal{D}(\sqrt{-\mathrm{L}})$, we have $\lVert \sqrt{-\mathrm{L}} f \rVert_{\pi}^2 = -\langle \mathrm{L}f,f \rangle_{\pi}$, which for $f \in \mathcal{C}^2_c(\R^d)$ is equal to $\lVert \nabla f \rVert_{\pi}^2$. We treat diffusion models satisfying the following conditions.

\begin{definition} 
\begin{enumerate}
\item[{\crtcrossreflabel{(PI)}[def:poinc]}] $X$ satisfies a \textit{Poincaré inequality}  with constant $C_{\mathrm{P}}$ if, for any $f \in \mathcal{D}(\sqrt{-\mathrm{L}})$, 
\[\mathrm{Var}_{\pi}(f) \coloneqq \pi(f^2) - \pi(f)^2 \leq C_{\mathrm{P}}\lVert \sqrt{-\mathrm{L}} f \rVert_{\pi}^2.\]
\item[{\crtcrossreflabel{(NI)}[def:nash]}] $X$ satisfies a \textit{Nash inequality} with constants $C^1_{\mathrm{N}}, C^2_{\mathrm{N}}$ if, for any $f \in \mathcal{D}(\sqrt{-\mathrm{L}})$, 
\[\lVert f \rVert^{d+2}_{\pi} \leq \Big(C^1_{\mathrm{N}} \lVert f \rVert^2_{\pi} + C^2_{\mathrm{N}} \lVert \sqrt{-\mathrm{L}} f \rVert^2_{\pi} \Big)^{d/2} \lVert f \rVert^2_{L^1(\pi)}.\]
\end{enumerate}
Denote by $\Sigma(C_{\mathrm{P}}, C^1_{\mathrm{N}}, C^2_{\mathrm{N}})$ the class of potentials $V\colon \R^d \to \R^d$ s.t.\ the corresponding Langevin diffusion satisfies \ref{def:poinc} and \ref{def:nash}.
\end{definition}
\begin{remark} 
A Nash inequality is called \textit{tight} if $C_{\mathrm{N}}^1 = 1$, in which case it implies a \textit{logarithmic Sobolev inequality} and therefore also a Poincaré inequality. 
\end{remark}
The combination of a Poincaré inequality with a Nash inequality is particularly attractive from a statistical point of view. A Poincaré inequality is equivalent to exponential ergodicity in $L^2(\pi)$. More precisely, for any $f \in L^2(\pi)$ with $\pi(f) = 0$, $\lVert P_t f \rVert_{\pi} \lesssim \exp(-t/C_{\mathrm{P}}) \lVert f \rVert_{\pi}$, which  enforces a fast mixing behavior of the diffusion. A Nash inequality on the other hand is equivalent to \textit{ultracontractivity} expressed through the heat kernel bound $\lVert P_t \rVert_{L^1(\pi) \to L^\infty(\pi)} \lesssim t^{-d/2}$ for $t \in (0,1]$, cf.\ \cite[Theorem 6.3.1]{bakry14}. The key observation for the statistical approach is that the combination of both yields tight variance bounds for path integrals $\int_0^t f(X_s) \diff{s}$ of functionals $f \in L^2(\pi)$, which allows for efficient control of the stochastic fluctuations of kernel estimators. 

To control the bias, we impose anisotropic Hölder regularity conditions on the invariant density. 

\begin{definition} 
Let $\bm{\beta} = (\beta_1,\ldots,\beta_d) \in (0,\infty)^d$, $\bm{\mathcal{L}} = (\mathcal{L}_1,\ldots,\mathcal{L}_d) \in (0,\infty)^d$. A function $h\colon \R^d \to \R$ is said to belong to the anisotropic Hölder class $\mathcal{H}_d(\bm{\beta},\bm{\mathcal{L}})$ if, for all $i=1,\ldots,d$, 
\begin{align*} 
\lVert D^k_i g \rVert_\infty &\leq \mathcal{L}_i, \quad k = 1,\ldots, \lfloor \beta_i \rfloor,\\ 
\lVert D_i^{\lfloor \beta_i \rfloor} g(\cdot + te_i) - D_i^{\lfloor \beta_i \rfloor} g(\cdot) \rVert_\infty &\leq \mathcal{L}_i \lvert t \rvert^{\beta_i - \lfloor \beta_i \rfloor}, \quad t \in \R,
\end{align*}
where $\lfloor \beta \rfloor$ denotes the largest integer strictly smaller than $\beta > 0$. Denote by 
\[\mathbb{H}_d(\bm{\beta},\bm{\mathcal{L}}) = \mathbb{H}_d(\bm{\beta},\bm{\mathcal{L}};\, C_\infty, C_{\mathrm{P}}, C_{\mathrm{N}}^1, C_{\mathrm{N}}^2),\] 
the set of invariant densities $\rho_V \in \mathcal{H}_d(\bm{\beta} + \one,\bm{\mathcal{L}})$ s.t.\ $\lVert \rho_V \rVert_\infty \leq C_\infty$ and $V \in \Sigma(C_{\mathrm{P}}, C^1_{\mathrm{N}}, C^2_{\mathrm{N}})$.
\end{definition}

Let $K\colon \R \to \R$ be a symmetric Lipschitz kernel function with $\supp(K) \subset [-1/2,1/2]$ and $\int K(x) \diff{x} = 1$. We say that $K$ is of order $\ell \in \N$ if $\int x^m K(x) \diff{x} = 0$ for any $m = 0,1,\ldots,\ell$. For $h > 0$ we let $K_h(\cdot) \coloneqq h^{-1} K(\cdot/h)$ and for $\bm{h},\bm{\eta} \in (0,\infty)^d$ we set 
\[\mathbb{K}_{\bm{h}}(x) \coloneqq \prod_{i=1}^d K_{h_i}(x_i), \quad x \in \R^d,\] 
and 
\[\mathbb{K}_{\bm{h}} \star \mathbb{K}_{\bm{\eta}}(x) \coloneqq \prod_{i=1}^d K_{h_i} \ast K_{\eta_i}(x_i), \quad x \in \R^d.\] 
We now define the following kernel estimators given a continuous record $(X_t)_{t \in [0,T]}$ s.t.\ $\rho_V \in \mathbb{H}(\bm{\beta},\bm{L})$: 
\[\hat{\rho}_{\bm{h},T}(x) \coloneqq \frac{1}{T}\int_0^T \mathbb{K}_{\bm{h}}(x-X_s) \diff{s}, \quad \hat{\rho}_{\bm{h},\bm{\eta},T}(x) \coloneqq \frac{1}{T}\int_0^T \mathbb{K}_{\bm{h}} \star \mathbb{K}_{\bm{\eta}}(x-X_s) \diff{s}, \qquad x \in \R^d.\]
In order to efficiently estimate $\rho_V$ via $\hat{\rho}_{\bm{h},T}$, the bandwidth $\bm{h}$ has to be carefully chosen to achieve an optimal balance between bias and variance of the kernel estimator. If the Hölder smoothness parameter $\bm{\beta}$ is unknown, the bias cannot be evaluated directly, which poses the fundamental challenge to design a fully data-driven/adaptive bandwith selection procedure to obtain a rate-optimal but possibly random bandwidth $\hat{\bm{h}}_T$.

As it turns out, in dimension $d = 2$ this problem is significantly simplified since in this case a tight variance bound of the kernel estimator only depends logarithmically on the bandwidth. Consequently, the smoothness independent and deterministic bandwith choice 
\[\hat{\bm{h}}_T \sim T^{-1/2}(1,1),\] 
yields the optimal $\sup$-norm estimation rate $\log T/\sqrt{T}$ given that the order of $K$ is chosen large enough. In dimension $d \geq 3$, the situation is significantly more involved and the bandwidth is chosen according to the following Lepski type selection rule:

Let $q \geq 1$ and  the set of candidate bandwidths $\mathcal{H}_T = \mathcal{H}_T^{(q)}$ be given by 
\[\mathcal{H}_T \coloneqq \Big\{\bm{h} \in (0,1]^d: T\mathfrak{a}_\circ^2  \geq \prod_{j=1}^d h_j^{(2/d)-1}\log T \Big\}.\] 
Here, $\mathfrak{a}_{\circ} = \mathfrak{a}_\circ^{(q)} \coloneqq (2 \Lambda)^{-2}$, where $\Lambda = \Lambda^{(q)} \coloneqq \gamma^\circ_{2q}(d,\lVert K \rVert_\infty;\, C_{\mathrm{P}}, C^1_{\mathrm{N}}, C^2_{\mathrm{N}})$ and the functional $\gamma^\circ_p$, $p \geq 1$, is defined in \cite[Remark 3.5]{str18} based on the functional $\gamma_p$ in \cite{lepski13}. Denote 
\[\overline{\varsigma}_T \coloneqq 2\Big(1\vee \sup_{\bm{h} \in \mathcal{H}_T} \Big\lVert T^{-1} \int_0^T \lvert  \mathbb{K}_{\bm{h}}(X_s - \cdot) \rvert \diff{s}\Big\rVert_\infty \Big),\]
and set 
\[\hat{\Delta}_T(\bm{h}) \coloneqq \sup_{\bm{\eta} \in \overline{\mathcal{H}}_T} \big\{\big[\lVert \hat{\rho}_{\bm{h},\bm{\eta}} - \hat{\rho}_{\bm{\eta}} \rVert_\infty - \lambda \hat{A}_T(\bm{\eta}) \big]_+ \big\},\] 
where $\overline{\mathcal{H}}_T \subset \mathcal{H}_T$ is the dyadic grid in $\mathcal{H}_T$, $\lambda = \lambda^{(q)} \coloneqq (1 \vee \lVert K \rVert_{L^1(\lebesgue)}^d) \Lambda$, and 
\[\hat{A}_T(\bm{h}) \coloneqq \prod_{i=1}^d h_i^{1/d - 1/2} \sqrt{\frac{\overline{\varsigma}_T \log T}{T}}.\] 
We now specify the bandwidth $\hat{\bm{h}}_T = \hat{\bm{h}}_T^{(q)}$ by 
\[\hat{\Delta}_T(\hat{\bm{h}}_T) + \lambda \hat{A}_T(\hat{\bm{h}}_T) =\inf_{\bm{h} \in \overline{\mathcal{H}}_T} \big\{\hat{\Delta}_T(\bm{h}) + \lambda \hat{A}_T(\bm{h}) \big\}.\]
Finally, for $q \geq 1$, $d \geq 2$ and $\hat{\bm{h}}_T = \hat{\bm{h}}_T^{(q)}$ given as above, we set 
\[\hat{\rho}_{\hat{\bm{h}}_T,T}(x) = \frac{1}{T}\int_0^T \mathbb{K}_{\hat{\bm{h}}_T}(x - X_s) \diff{s}, \quad x \in \R^d.\] 
According to the discussion in \cite[Section 3]{str18} on the two-dimensional case and \cite[Theorem 3.4]{str18}, we now have the following uniform $\sup$-norm estimation  result. 

\begin{theorem}\label{theo:est}
Suppose $\bm{\beta} \in (0,\mathfrak{b}]^d$ for some $\mathfrak{b} \in \N \cap[2,\infty)$ and let $K$ have order $\mathfrak{b}+1$. Then, for any $q \geq 1$ and $\bm{\mathcal{L}} \in (0,\infty)^d$ it holds 
\[\sup_{\rho_V \in \mathbb{H}_d(\bm{\beta},\bm{\mathcal{L}})} \Big(\E^\pi\big[\big\lVert\hat{\rho}_{\hat{\bm{h}}_T,T} - \rho_V \big\rVert_\infty^q\big] \Big)^{1/q} = \mathcal{O}(\Psi_{d,\bm{\beta}}(T)),\] 
where for the harmonic mean smoothness $\overline{\bm{\beta} + \one} \coloneqq (d^{-1}\sum_{i=1}^d \frac{1}{\beta_i +1})^{-1}$, the rate $\Psi_{d,\bm{\beta}}$ is specified by 
\[\Psi_{d,\bm{\beta}}(T) \coloneqq \begin{cases} \frac{\log T}{\sqrt{T}}, &d =2,\\ \Big(\frac{\log T}{T}\Big)^{\frac{\overline{\bm{\beta} + \one}}{2\overline{\bm{\beta}+ \one}+d-2}}, &d \geq 3. \end{cases}\]
\end{theorem}

\subsection{Data-driven estimation of the optimal reflection boundary}
In this section we consider a set of domains $\Theta \subset \mathbf{D}$ such that the set of minimizers $\argmin_{D \in \Theta} J(D)$ is well-defined and let $D^\ast \in \argmin_{D \in \Theta} J(D)$. Our data-driven procedure to determine reflection domains $\hat{D}$ whose average costs are close to the optimal ergodic costs $J(D^\ast)$ uses the following assumption.

\begin{assumption} \label{ass:stat}
\begin{enumerate}[label = (\roman*), ref = (\roman*)]
\item \label{ass:stat1} For some constants $\underline{\lambda}, \overline{\lambda},\Lambda$ it holds $B(0,\underline{\lambda}) \subset D^\ast \subset B(0,\overline{\lambda})$ and $\mathcal{H}^{d-1}(\partial D^\ast) \leq \Lambda$; 
\item we are given information on two constants  $\underline{\rho},\overline{\rho}$ such that $\underline{\rho} \leq \inf_{B(0,\overline{\lambda})} \rho \leq \sup_{B(0,\overline{\lambda})} \rho \leq \overline{\rho}$.
\end{enumerate}
\end{assumption}

Accordingly, we define the truncated invariant density estimator $\hat{\rho}^\ast_{T,q}$ based on data $(X_t)_{t \in [0,T]}$ of the uncontrolled diffusion process as 
\[\hat{\rho}^\ast_{T,q}(x) \coloneqq (\hat{\rho}_{T,q}(x) \wedge 2\overline{\rho}) \vee \underline{\rho}/2, \quad x \in \R^d,\] 
with $\hat{\rho}_{T,q}$ specified as the adaptive invariant density estimator from the previous subsection with bandwidth choice $\bm{h} = \hat{\bm{h}}_T^{(q)}$ for $q \geq 1$. Moreover, we let $\Theta(\underline{\lambda},\overline{\lambda},\Lambda) \subset \Theta$ be the subfamily of reflection domains satisfying Assumption \ref{ass:stat}.\ref{ass:stat1}. Let 
\[\hat{J}_{T,q}(D) \coloneqq \frac{1}{\int_D \hat{\rho}^\ast_{T,q}(x) \diff{x}}\Big(\int_D f(y) \hat{\rho}^\ast_{T,q}(y) \diff{y} + \kappa \int_{\partial D} \hat{\rho}_{T,q}^\ast(y) \, \mathcal{H}^{d-1}(\diff{y}) \Big), \quad D \in \Theta,\] 
be the estimator of the asymptotic costs associated to the reflection domain $D$ and define the reflection domain estimator 
\[\hat{D}_{T,q} \in \argmin_{D \in \Theta(\underline{\lambda},\overline{\lambda},\Lambda)} \hat{J}_{T,q}(D).\]
Here, we must assume that $\Theta \subset \mathbf{D}$ is a metrizable space that is sufficiently nice to allow a measurable choice of $\hat{D}_{T,q}$ considered as a random mapping into the  Borel space associated to $\Theta$. 
We now have the following concentration result for the simple regret.

\begin{proposition} \label{prop:stat_regret}
Suppose that $X_0 \sim \mu$, where $\mu \ll \pi$ with $\lVert \tfrac{\diff{\mu}}{\diff{\pi}} \rVert_{L^q(\pi)} < \infty$ for some $q \in (1,\infty]$. Then, for any $p \geq 1$, given the assumptions from Theorem \ref{theo:est} we have the regret bound
\[\E^\mu\big[\big\lvert J\big(D^\ast\big) - J\big(\hat{D}_{T,p\overline{q}}\big) \big\rvert^p \big]^{1/p} \leq C\lVert \tfrac{\diff{\mu}}{\diff{\pi}} \rVert_{L^q(\pi)}^{1/p} \Psi_{d,\bm{\beta}}(T),\]
where $\overline{q} \coloneqq q/(q-1)$ is the conjugate Hölder exponent of $q$ and $C$ depends on $\kappa,p, q, f$ and the constants from Assumption \ref{ass:stat}. 
\end{proposition}
\begin{proof} 
For fixed $D \in \Theta(\underline{\lambda},\overline{\lambda},\Lambda)$, write (in obvious notation) 
\[J(D) = A(D)/B(D), \quad \hat{J}_{T,p\overline{q}}(D) = \hat{A}_T(D)/\hat{B}_T(D).\]
Using 
\[B(D) \wedge \hat{B}_T(D) \geq \int_{B(0,\underline{\lambda})} \rho(x) \wedge \hat{\rho}^\ast_{T,p\overline{q}}(x) \diff{x} \geq \frac{\underline{\rho}}{2} \lebesgue(B(0,\underline{\lambda})) \eqqcolon \varpi,\] 
it follows 
\begin{align*} 
\big\lvert J(D) - \hat{J}_{T,p\overline{q}}(D) \big\rvert &\leq \frac{A(D)}{B(D)\hat{B}_T(D)} \big\lvert \hat{B}_T(D) - B(D) \big\rvert + \frac{\lvert \hat{A}_T(D)- A(D)\rvert}{\hat{B}_T(D)}\\ 
&\leq \frac{\lVert f \rVert_{L^1(\pi)} + \kappa \overline{\rho} \Lambda}{\varpi^2}\lVert \rho - \hat{\rho}^\ast_{T,p\overline{q}} \rVert_{L^\infty(B(0,\overline{\lambda}))} + \frac{\lVert f \rVert_{L^1(B(0,\overline{\lambda}))} + \kappa\Lambda}{\varpi} \lVert \rho - \hat{\rho}^\ast_{T,p\overline{q}} \rVert_{L^\infty(B(0,\overline{\lambda}))}.
\end{align*}
Thus, 
\begin{equation}\label{eq:uni_regret}
\begin{split}
\E^\mu\Big[\sup_{D \in \Theta(\underline{\lambda},\overline{\lambda},\Lambda)} \lvert J(D) - \hat{J}_{T,p\overline{q}}(D) \rvert^p \Big]^{1/p} &\leq C(\kappa,\underline{\lambda},\overline{\lambda},\underline{\rho},\overline{\rho},p, f) \E^\mu\big[\lVert \rho - \hat{\rho}^\ast_{T,p\overline{q}}\rVert_{L^\infty(B(0,\overline{\lambda}))}^p \big]^{1/p}\\ 
&\leq C(\kappa,\underline{\lambda},\overline{\lambda},\underline{\rho},\overline{\rho},p, f) \lVert \tfrac{\diff{\mu}}{\diff{\pi}} \rVert_{L^q(\pi)}^{1/p} \E^\pi\big[\lVert \rho - \hat{\rho}^\ast_{T,p\overline{q}}\rVert_{L^\infty(B(0,\overline{\lambda}))}^{p\overline{q}} \big]^{1/p\overline{q}}, 
\end{split}
\end{equation}
where we used Hölder inequality twice for the last line. The expectation can be bounded as follows:
\begin{align*} 
\E^\pi\big[\lVert \rho - \hat{\rho}^\ast_{T,p\overline{q}}\rVert_{L^\infty(B(0,\overline{\lambda}))}^{p\overline{q}} \big]^{1/p\overline{q}} &\leq \E^\pi\big[\lVert \rho - \hat{\rho}_{T,p\overline{q}}\rVert_\infty^{p\overline{q}} \big]^{1/p\overline{q}} + \E^\pi\big[\lVert \rho - \hat{\rho}^\ast_{T,p\overline{q}}\rVert_{L^\infty(B(0,\overline{\lambda}))}^{p\overline{q}}  \one_{\{\lVert \hat{\rho}_{T,p\overline{q}} - \rho \rVert_\infty > \underline{\rho}/2\}}\big]^{1/p\overline{q}}\\ 
&\leq \E^\pi\big[\lVert \rho - \hat{\rho}_{T,p\overline{q}}\rVert_\infty^{p\overline{q}} \big]^{1/p\overline{q}} + 2\overline{\rho}\big(\PP^\pi\big(\lVert \hat{\rho}_{T,p\overline{q}} - \rho \rVert_\infty > \underline{\rho}/2 \big)\big)^{1/p\overline{q}} \\ 
&\leq (1+4\overline{\rho}/\underline{\rho})  \E^\pi\big[\lVert \rho - \hat{\rho}_{T,p\overline{q}}\rVert_\infty^{p\overline{q}} \big]^{1/p\overline{q}} \\
&\lesssim (1+4\overline{\rho}/\underline{\rho}) \Psi_{d,\bm{\beta}}(T).
\end{align*}
For the first inequality we used that on $\{\lVert \hat{\rho}_{T,p\overline{q}} - \rho \rVert_\infty \leq \underline{\rho}/2\}$ it holds $\hat{\rho}_{T,p\overline{q}}\vert_{B(0,\overline{\lambda})} \equiv \hat{\rho}^\ast_{T,p\overline{q}}\vert_{B(0,\overline{\lambda})}$. The last two lines follow from the Markov inequality and Theorem \ref{theo:est}. Plugging this bound into \eqref{eq:uni_regret} we find 
\begin{equation}\label{eq:uni_regret2} 
\E^\mu\Big[\sup_{D \in \Theta(\underline{\lambda},\overline{\lambda},\Lambda)} \lvert J(D) - \hat{J}_{T,p\overline{q}}(D) \rvert^p \Big]^{1/p} \lesssim \lVert \tfrac{\diff{\mu}}{\diff{\pi}} \rVert_{L^q(\pi)}^{1/p} \Psi_{d,\bm{\beta}}(T).
\end{equation}
Finally, since $D^\ast \in \argmin_{D \in \Theta(\underline{\lambda},\overline{\lambda},\Lambda)} J(D)$ and $\hat{D}_{T,p\overline{q}} \in \argmin_{D \in \Theta(\underline{\lambda},\overline{\lambda},\Lambda)} \hat{J}_{T,p\overline{q}}(D)$, we have 
\[\E^\mu\big[\big\lvert J\big(D^\ast\big) - J\big(\hat{D}_{T,p\overline{q}}\big) \big\rvert^p \big]^{1/p} \leq 2 \E^\mu\Big[\sup_{D \in \Theta(\underline{\lambda},\overline{\lambda},\Lambda)} \lvert J(D) - \hat{J}_{T,p\overline{q}}(D) \rvert^p \Big]^{1/p},\] 
which in combination with \eqref{eq:uni_regret2} yields the claim.
\end{proof}

This result may be interpreted in two different ways. On the one hand, it shows  for the generic situation, where the controller has access to a separate diffusion data sample and uses it in online estimation of an optimal reflection boundary, that the regret vanishes at the nonparametric estimation rate. On the other hand, it demonstrates that a simple explore-then-commit strategy, where we first estimate an optimal set for $T$ time units and afterwards exploit by reflecting the process at the estimated boundaries, yields a regret  bounded by $\Psi_{d,\beta}(T)$. 

As a proof of concept, we apply the above methodology to simulated data.
In particular, we consider the Ornstein--Uhlenbeck process $X$ in $\R^2$ governed by
\[
	\mathrm{d}X_t
	=-\frac{X_t}{10}\, \mathrm{d}t+\sqrt{2}\,\mathrm{d}W_t,
\]
and simulate data from this model until time $T_{\mathrm{end}}$ for increasing values of $T_{\mathrm{end}}$, corresponding to increasing periods of exploration.
For each simulation, we then estimate the invariant density $\rho$ (here a normal density) via a discretized version of the kernel density estimator as well as its gradient $\nabla\rho$ via the gradient of the kernel estimator. 
Finally, setting $f=\n{\cdot}$ and $\kappa=1$, we use these estimates along with Corollary \ref{coro:twodim} to find approximately optimal star-shaped polygons, and plot them in Figure \ref{fig:kernel_estimates} along with the numerically approximated optimal shape, i.e. the one found using Corollary \ref{coro:twodim} with the true density $\rho$.
\begin{figure}[h]
        \label{fig:kernel_estimates}
	\centering
	\includegraphics[width=.55\textwidth]{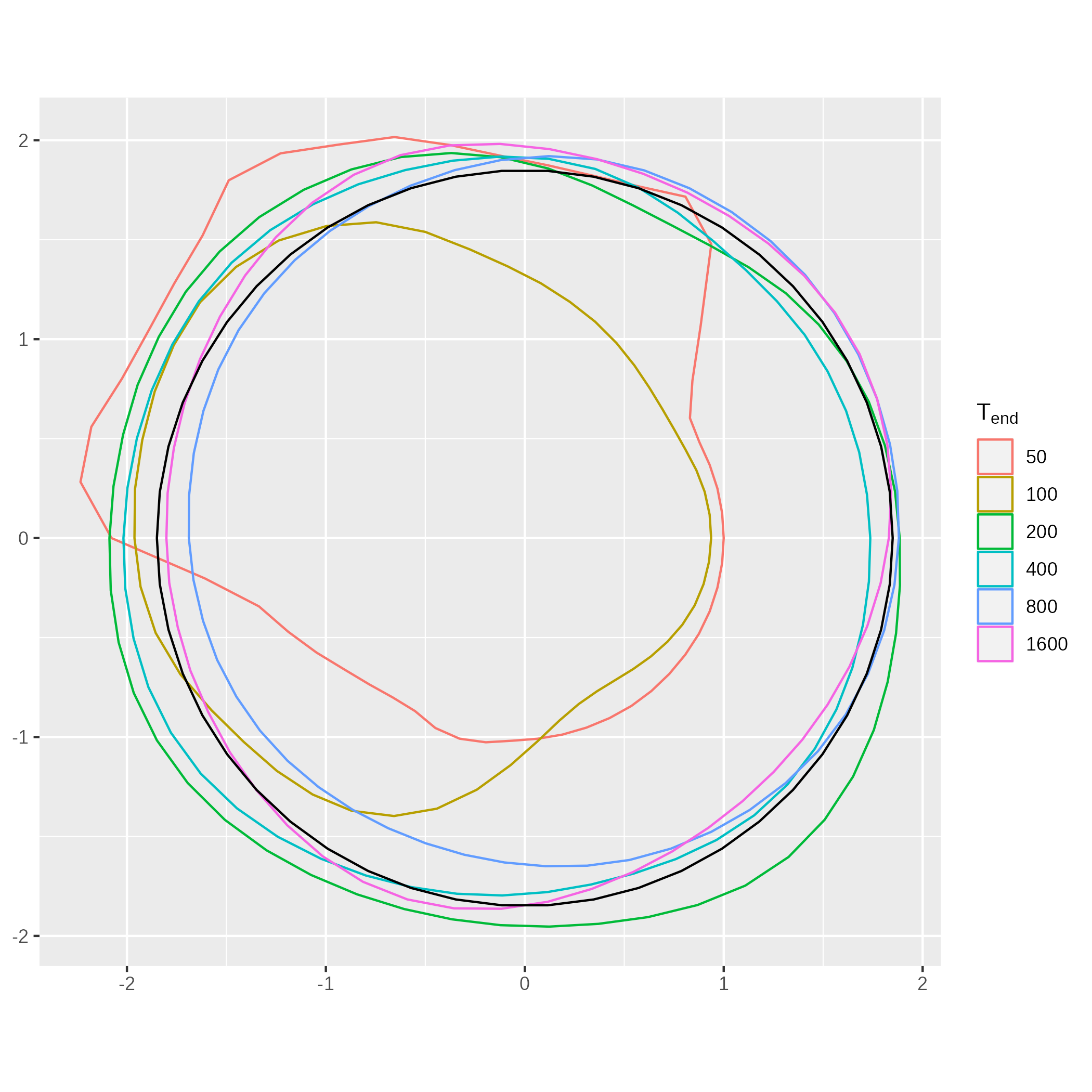}
	\caption{Left: Estimates of the optimal shape (black) using kernel estimates after increasing periods of exploration. Notably, after only $T=150$, the estimated optimal shape has an associated cost only $0.61\%$ higher than the true optimum.}
\end{figure}

\paragraph{Episodic domain learning}
Aiming now at strategies with sublinear regret rates without any simplifying assumptions on the data-collection mechanism, we face a classical exploration vs.\ exploitation dilemma in light of the necessity to simultaneously control the process and estimate its dynamics over time. The key bound from Proposition \ref{prop:stat_regret} will allow us to do so. 

Our episodic learning algorithm separates the time-line into exploration and exploitation phases. Let $T_i$ be the start of the $i$-th exploration period, where we let the diffusion run freely without reflection and $S_i$ be the start of the $i$-th exploitation period, where we reflect the process according to an estimate of the optimal reflection boundary based on past observations of the exploration process. We always start with an exploration period, i.e.\ $T_1 = 0$, and then alternate between exploration and exploitation periods.  We denote by $\tau_i = S_i - T_i$ the length of the $i$-th exploration period and by $\sigma_i = T_{i+1} -  S_i$ the length of the $i$-th exploitation period.

Contrary to the scalar diffusion case in \cite{christensen21_alt}, the multivariate diffusion does not hit points, which makes it difficult to introduce an appropriate life-cycle decomposition of the exploration process that allows for elegant renewal theoretic arguments in the analysis. Instead we choose  sequences $(a_i) \subset [1,\infty)^{\N}$, $(b_i) \subset [1,\infty)^{\N}$ and simply let $S_i = \inf\{t \geq T_i + a_i: \tilde{X}_t \in \operatorname{cl}(B(0,\underline{\lambda}))\}$ and $T_{i+1} = S_i + b_i$, where $(\tilde{X}_t)_{t \geq 0}$ denotes the process that is controlled according to the above strategy. This implies that the $i$-th exploitation length is deterministically given by $\sigma_i = b_i$ and due to the strongly recurrent behavior of the process the $i$-th exploration length is relatively close to $a_i$ with high probability. Moreover, the strategy makes sure that at the start of an exploitation period we have $\tilde{X}_{S_i} \in \operatorname{cl} B(0,\underline{\lambda}) \subset \overline{D}$ for any $D \in \Theta(\underline{\lambda},\overline{\lambda},\Lambda)$. For the estimator of the reflection boundary $\hat{D}_i$ in the $i$-th exploitation period, we only take into account the observations gathered from the last exploration period on the time interval $[T_i,T_i+a_i]$ by letting 
\[\hat{\rho}^\ast_i(x) \coloneqq (\hat{\rho}_{i,2}(x) \wedge 2 \overline{\rho}) \vee \underline{\rho}/2, \quad x \in \R^d,\]
where $\hat{\rho}_{i,2}$ is the adaptive invariant density estimator based on the diffusion data $(\tilde{X}_t)_{t \in [T_i,T_i+a_i]}$ for the parameter choice $q=2$ in the construction of the stochastic bandwidth, and then set 
\[\hat{D}_i \in \argmin_{D \in \Theta(\underline{\lambda},\overline{\lambda},\Lambda)} \frac{1}{\int_D \hat{\rho}^\ast_{i}(y) \diff{y}}\Big(\int_D f(y) \hat{\rho}^\ast_{i}(y) \diff{y} + \kappa \int_{\partial D} \hat{\rho}_{i}^\ast(y) \, \mathcal{H}^{d-1}(\diff{y}) \Big).\]

Let $\tilde{\mathbb{F}} = (\tilde{\mathcal{F}}_t)_{t \geq 0}$ be the filtration generated by the controlled process $\tilde{X}$ and set $\tilde{C}_{a,b} \coloneqq \int_a^b f(\tilde{X}_s) \diff{s} + \kappa (\tilde{L}_b - \tilde{L}_a)$ as the costs on the time interval $[a,b]$ associated to  $\tilde{X}$, where $\tilde{L}$ is the local time on the reflection boundaries during the exploitation phases and is set equal to zero during the exploration phases. We also let $C_{a,b}(x,D)$ be the costs on the time interval $[a,b]$ associated to a Langevin diffusion $Z^{x,D}$ that is driven by a Brownian motion independent of $\tilde{\mathbb{F}}$ and that is reflected in $D$ and is started in $x$. Denote by $\tau(x,D)$ its first hitting time of $\operatorname{cl}(B(0,\underline{\lambda}))$. Furthermore, we set $n(T) \coloneqq \min\{i \in \N: \sum_{j=1}^i (a_j + b_j) \geq T\}$ and note that 
\[n(T) \geq \min\Big\{i\in \N: \sum_{j=1}^{i} (\tau_j + \sigma_j) \geq T\Big\} = \min\{i\in \N: T_{i+1} \geq T\} = \min\{i\in \N: T_i \geq T\} -1  ,\]
which in particular implies that $S_{n(T)+1} > T_{n(T)+1} \geq T$.

For technical reasons, we assume that the potential $V$ satisfies the following drift condition: for some constants $r,M > 0$ it holds 
\begin{equation}\label{cond:drift}
\forall \lvert x \rvert \geq M:\,  \langle \nabla V(x), x\slash \lvert x \rvert \rangle \geq r.
\end{equation}
Due to the specific structure of the generator, it is well known, see e.g.\ \cite{bakry08}, that the Langevin diffusion then satisfies a Poincaré inequality \ref{def:poinc} and that its generator has a Lyapunov function $\mathrm{V} \geq 1$ that is locally bounded, is for some $a,R > 0$  given by $\mathrm{V}(x) =\exp(a \lvert x \rvert)$ for  all $\lvert x \rvert \geq R$, and satisfies $\pi(\mathrm{V}) < \infty$. This implies, see \cite[Theorem 5.2, Theorem 7.2]{DownMeynTweedie1995}, that $X$ is $\mathrm{V}$-uniformly ergodic in the sense that for some constant $b > 0$,
\begin{equation}\label{eq:v-erg}
\sup_{\lvert g \rvert \leq \mathrm{V}} \lvert P_tg(x) - \pi(g)\rvert \lesssim \mathrm{V}(x)\exp(-bt), \quad x \in \R^d,
\end{equation}
and that for any set $C$ s.t.\ $\lebesgue(C) > 0$, it holds 
\begin{equation}\label{eq:mod_mom} 
\E^x\Big[\int_0^{\tau_C} \mathrm{V}(X_s) \diff{s}\Big] \leq c(C) \mathrm{V}(x), \quad x \in \R^d,
\end{equation} 
where $c(C)$ is a constant depending on $C$ and $\tau_C$ is the first hitting time of $C$ (we note that a much stronger statement than \eqref{eq:mod_mom} is true, but this won't be needed in our context). We also need some assumptions on the set of viable reflection domains $\Theta(\underline{\lambda},\overline{\lambda},\Lambda)$ that allow sufficient uniform bounds in the following. More precisely, we assume  the constants $C^\prime(D)$ appearing in Theorem \ref{prop:erg_cost} to be uniformly bounded in $D$, that is, 
\begin{equation}\label{eq:uni_const}
\sup_{D \in \Theta(\underline{\lambda},\overline{\lambda},\Lambda)} C^\prime(D) < \infty.
\end{equation} 
Note that this assumption boils down to uniform lower bounds on the transition densities, cf.\ Lemma \ref{lem:erg}, and uniform bounds on the maxima of the functions $\varphi_D$ and their partial derivatives described in Lemma \ref{lem:uninormal}, i.e., certain uniform regularity assumptions on the boundaries. Moreover, we also require  a uniform upper bound on the transition densities in the form  
\begin{equation}\label{eq:uni_trans}
\sup_{D \in \Theta(\underline{\lambda},\overline{\lambda},\Lambda)} \sup_{t \geq 1, x,y \in \overline{D}} p^D_t(x,y) < \infty.
\end{equation}
By continuity of $p_t^D$, it follows from the Gaussian upper bound on $p^D_t(x,y)$ given in \cite[Corollary 6.15]{ouh05} for a.e.\ $(x,y) \in D^2$ that this holds pointwise, i.e., $\sup_{t \geq 1, x,y \in \overline{D}} p^D_t(x,y) < \infty$ for any $D \in \Theta(\underline{\lambda},\overline{\lambda},\Lambda)$.  Verifying \eqref{eq:uni_const} and \eqref{eq:uni_trans} is highly problem specific and is a difficult task with the tools available in the literature when $\Theta(\underline{\lambda},\overline{\lambda},\Lambda)$ is infinite. Still, these assumptions are not unreasonable provided that   appropriate uniform regularity conditions on the boundaries are in force. With this technical preparation we can prove our final theorem.

\begin{theorem} \label{theo:expl}
Suppose that the non-reflected Langevin diffusion satisfies \eqref{cond:drift} and  \ref{def:nash} and that its invariant density satisfies $\rho_V \in \mathcal{H}_d(\bm{\beta} + \one,\bm{\mathcal{L}})$. Assume also that $f \in L^2(\pi)$, $f \lesssim \mathrm{V}$,  that  the initial distribution of the first exploration phase $\mu$ satisfies $\mu \ll \pi$ and $\diff\mu/\diff\pi \in L^2(\pi)$ and that \eqref{eq:uni_const} and \eqref{eq:uni_trans} hold. Then, the average regret per time unit is bounded by 
\[\frac{1}{T}\E[\tilde{C}_{0,T}] - J(D^\ast) \lesssim \frac{1}{T} \Big(\sum_{i=1}^{n(T)} a_i + \sum_{i=1}^{n(T)} b_i \Psi_{d,\bm{\beta}}(a_i)\Big).\]
\end{theorem}
\begin{proof}
Without loss of generality, let $T \geq 1$. Using that the costs are nonnegative, we have  
\begin{equation} \label{eq:decomp}
\begin{split}
\E[\tilde{C}_{0,T}] &= \E\Big[\sum_{T_i \leq T} \tilde{C}_{T_i, S_i \wedge T}\Big] + \E\Big[\sum_{S_i \leq T} \tilde{C}_{S_i, T_{i+1} \wedge T}\Big]\\ 
&\leq \sum_{i=1}^{n(T)} \E[\tilde{C}_{T_i,S_i}] + \sum_{i=1}^{n(T)} \E[\tilde{C}_{S_i \wedge T,T_{i+1}\wedge T}] \\
&\leq \sum_{i=1}^{n(T)} \big(\E\big[C_{0,a_i}(\tilde{X}_{T_{i}}, \R^d)\big] + \E\big[C_{0,\tau(\tilde{X}_{T_i +a_i},\R^d)}(\tilde{X}_{T_{i}+a_i}, \R^d)\big]\big) + \sum_{i=1}^{n(T)-1} \E\big[C_{0,b_i}(\tilde{X}_{S_{i}}, \hat{D}_i)\big] \\ 
&\qquad + \E\big[C_{0,(T-S_{n(T)}) \vee 1}(\tilde{X}_{S_{n(T)}}, \hat{D}_{n(T)})\big].
\end{split}
\end{equation}
We start by bounding the second term, associated to the exploitation periods until the $(n(T)-1)$-th episode. By conditioning on $\tilde{\mathcal{F}}_{S_i}$ we see that 
\begin{equation}\label{eq:exploit0}
\begin{split}
&\Big\lvert \E\big[C_{0,b_i}(\tilde{X}_{S_{i}}, \hat{D}_i)\big] - b_i \E[J(\hat{D}_i) \big] \Big\rvert\\
&\,\leq b_i   \int_{\Theta(\underline{\lambda},\overline{\lambda},\Lambda)} \int_{\operatorname{cl} B(0,\underline{\lambda})} \Big\lvert\E^x\Big[ \frac{1}{b_i}\int_0^{b_i} f(X_s^D) \diff{s} + \kappa L^D_{b_i} - J(D) \Big]  \Big\rvert \,\PP(\tilde{X}_{S_i} \in \diff{x}, \hat{D}_i \in \diff D)  \\
&\,\leq \sup_{D \in \Theta(\underline{\lambda},\overline{\lambda},\Lambda)} C^\prime(D) < \infty, 
\end{split}
\end{equation} 
where the last two lines follow from Theorem \ref{prop:erg_cost} and \eqref{eq:uni_const}. Observe now that on the previous data collection interval $[T_i, T_i+a_i]$ the process $\tilde{X}$ is equal in law to the Langevin diffusion with potential $V$ started according to the law $\PP_{\tilde{X}_{T_i}}$. For $i=1$, the latter has, by assumption on $\mu$, a Radon--Nikodym derivative w.r.t.\ the invariant distribution $\pi$ that lies in $L^2(\pi)$. Extend the transition densities $p_t^D$ from $\overline{D}^2$ to $\R^d \times \R^d$ by setting $p_t^D(x,y) = 0$ for $x,y \notin \overline{D}$. For $i \geq 2$, we then observe that  for any $y \in \operatorname{cl}(B(0,\overline{\lambda}))$ we have
\begin{align*} 
\frac{\diff{\PP(\tilde{X}_{T_i}\in \cdot)}}{\diff{\lebesgue}}(y) &=  \int_{\Theta(\underline{\lambda},\overline{\lambda},\Lambda)} \int_{\operatorname{cl} B(0,\underline{\lambda})} p^D_{b_i}(x,y) \, \PP(\tilde{X}_{S_{i-1}} \in \diff{x}, \hat{D}_i \in \diff{D})\\ 
&\leq \sup_{D \in \Theta(\underline{\lambda},\overline{\lambda},\Lambda)} \sup_{t \geq 1, (x,y) \in \overline{D}^2} p^D_t(x,y) < \infty,
\end{align*}
where we used \eqref{eq:uni_trans}. 
Since the Lebesgue density of $\pi$ is bounded away from zero on $B(0,\overline{\lambda})$, this now implies 
\begin{equation}\label{eq:unif_initial}
\sup_{i \in \N}\big\lVert \tfrac{\diff{\PP(\tilde{X}_{T_i} \in \cdot)}}{\diff{\pi}} \big\rVert_{L^2(\pi)} < \infty.
\end{equation}
Proposition \ref{prop:stat_regret} therefore yields that 
\[\E\big[\big\lvert J(D^\ast) - J (\hat{D}_i)\big\rvert\big] \lesssim \Psi_{d,\bm{\beta}}(a_i).\] 
In summary, the above estimates and the triangle inequality therefore yield 
\begin{equation} \label{eq:bound_exploit}
\begin{split}
\sum_{i=1}^{n(T)-1} \E\big[C_{0,b_i}(\tilde{X}_{S_{i}}, \hat{D}_i)\big] &\lesssim n(T) +  \sum_{i=1}^{n(T)-1} b_i (\Psi_{d,\bm{\beta}}(a_i) + J(D^\ast)).
\end{split}
\end{equation}
We turn to the last summand in \eqref{eq:decomp}. Similarly to \eqref{eq:exploit0} we see that 
\begin{equation}\label{eq:exploit1}
\begin{split}
&\Big\lvert \E\big[C_{0,(T- S_{n(T)}\vee 1}(\tilde{X}_{S_{n(T)}}, \hat{D}_{n(T)})\big] - \E\big[((T- S_{n(T)})\vee 1)J(\hat{D}_{n(T)}) \big] \Big\rvert \leq \sup_{D \in \Theta(\underline{\lambda},\overline{\lambda},\Lambda)} C^\prime(D) < \infty, 
\end{split}
\end{equation} 
and using \eqref{eq:unif_initial} and Proposition \ref{prop:stat_regret} we have 
\begin{equation}\label{eq:exploit2}
\begin{split}
\E\big[((T-S_{n(T)} \vee 1)\big\lvert J(D^\ast) - J (\hat{D}_{n(T)})\big\rvert\big]&\leq b_{n(T)} \E\big[\big\lvert J(D^\ast) - J (\hat{D}_{n(T)})\big\rvert\big]
\lesssim b_{n(T)} \Psi_{d,\bm{\beta}}(a_{n(T)}).
\end{split}
\end{equation}
Combining \eqref{eq:bound_exploit}, \eqref{eq:exploit1} and \eqref{eq:exploit2} we finally arrive at 
\begin{equation}\label{eq:exploit3}
\begin{split}
&\sum_{i=1}^{n(T)-1} \E\big[C_{0,b_i}(\tilde{X}_{S_{i}}, \hat{D}_i)\big]  
+ \E\big[C_{0,(T-S_{n(T)}) \vee 1}(\tilde{X}_{S_{n(T)}}, \hat{D}_{n(T)})\big] \\
&\, \leq n(T) + \sum_{i=1}^{n(T)} b_i \Psi_{d,\bm{\beta}}(a_{n(T)}) + J(D^\ast)\Big(\sum_{i=1}^{n(T)-1} b_i + \E[(T-S_{n(T)})\vee 1]\Big) \\ 
&\leq n(T) + TJ(D^\ast) + \sum_{i=1}^{n(T)} b_i \Psi_{d,\bm{\beta}}(a_i).
\end{split}
\end{equation} 
Let us now treat the exploration periods. Recall that since the unreflected diffusion satisfies a Poincaré inequality, we have exponentially fast convergence of its semigroup in $L^2(\pi)$. Since moreover $f\in L^2(\pi)$, the combined statements of Theorem 3.1 and Corollary 3.2 in \cite{cattiaux12} imply that for $\tilde{f} \coloneqq f -\mu(f)$ and $g \coloneqq  -\int_0^\infty P_s \tilde{f} \diff{s}$, we have $g \in \mathcal{D}(\mathrm{L})$ and 
\[\Big\lVert \int_0^t (f - \pi(f))(X_s) \diff{s} \Big\rVert_{L^2(\PP^\pi)}^2 \leq C t\lVert \sqrt{-\mathrm{L}} g \rVert_\pi^2, \]
for some constant $C$ that is independent of $t \geq 0$.
Using \eqref{eq:unif_initial} and Hölder inequality therefore implies 
\begin{align*} 
\Big\lvert \E\big[C_{0,a_i}(\tilde{X}_{T_{i}}, \R^d)\big] - a_i\pi(f) \Big\rvert &= \Big\lvert\E\Big[\E^{\tilde{X}_{T_i}}\Big[\int_0^{a_i} (f - \pi(f))(X_s)\diff{s} \Big] \Big]  \Big\rvert\\ 
&\leq \int_{\R^d} \E^x\Big[\Big\lvert \int_0^{a_i} (f - \pi(f))(X_s)\diff{s} \Big\rvert \Big] \frac{\diff{\PP(\tilde{X}_{T_i} \in \cdot)}}{\diff{\pi}}(x)  \, \pi(\diff{x})\\ 
&\leq \Big\lVert \frac{\diff{\PP(\tilde{X}_{T_i} \in \cdot)}}{\diff{\pi}} \Big\rVert_{L^2(\pi)} \Big\lVert \int_0^{a_i} (f - \pi(f))(X_s) \diff{s} \Big\rVert_{L^2(\PP^\pi)}\\
&\leq C^\prime \sqrt{a_i}.
\end{align*}
It therefore follows by the triangle inequality that 
\begin{equation}\label{eq:bound_explore1}
\sum_{i=1}^{n(T)} \E\big[C_{0,a_i}(\tilde{X}_{T_{i}}, \R^d)\big] \lesssim \sum_{i=1}^{n(T)} a_i.
\end{equation}
Furthermore, using  \eqref{eq:mod_mom} and the assumption $f \lesssim \mathrm{V}$, we can write 
\[\E\big[C_{0,\tau(\tilde{X}_{T_i +a_i},\R^d)}(\tilde{X}_{T_{i}+a_i}, \R^d)\big] = \int_{\R^d} \E^x\Big[\int_0^{\tau_{\operatorname{cl}(B(0,\underline{\lambda}))}} f(X_s) \diff{s} \Big] \,\PP(\tilde{X}_{T_i+a_i} \in \diff{x}) \lesssim \E\big[\mathrm{V}(\tilde{X}_{T_i+a_i})\big].\]
Now since $\tilde{X}_{T_i+a_i}$ has the same law as $X_{a_i}$ with $X$ started according to the law $\PP_{\tilde{X}_{T_i}}$ which, by construction, is supported on $\operatorname{cl}(B(0,\overline{\lambda}))$, it follows from \eqref{eq:v-erg} that 
\begin{align*}
\E\big[\mathrm{V}(\tilde{X}_{T_i+a_i})\big] &\leq \sup_{x \in \operatorname{cl}(B(0,\overline{\lambda}))} \E^x[\mathrm{V}(X_{a_i})] \leq \pi(\mathrm{V}) + \sup_{x \in \operatorname{cl}(B(0,\overline{\lambda}))} \lvert P_{a_i} \mathrm{V}(x) - \pi(\mathrm{V}) \rvert\\ 
&\leq \pi(\mathrm{V}) + C\lVert \mathrm{V} \rVert_{L^\infty(\operatorname{cl}(B(0,\overline{\lambda})))} < \infty, 
\end{align*}
for some constant independent of $i \in \N$. This allows us to conclude that 
\begin{equation} \label{eq:bound_explore2}
\sum_{i=1}^{n(T)}\E\big[C_{0,\tau(\tilde{X}_{T_i +a_i},\R^d)}(\tilde{X}_{T_{i}+a_i}, \R^d)\big] \leq \big(\pi(\mathrm{V}) + C\lVert \mathrm{V}\rVert_{L^\infty(\operatorname{cl}(B(0,\overline{\lambda})))}\big)n(T).
\end{equation}
Taking together the bounds \eqref{eq:exploit3}, \eqref{eq:bound_explore1} and \eqref{eq:bound_explore2} now shows that the average regret per time unit is bounded by
\[\frac{1}{T}\E[\tilde{C}_{0,T}] - J(D^\ast) \lesssim \frac{1}{T} \Big(\sum_{i=1}^{n(T)} a_i + \sum_{i=1}^{n(T)} b_i \Psi_{d,\bm{\beta}}(a_i)\Big).\]
\end{proof}

As a consequence, we obtain the following explicit rates for a strategy that doubles exploration times and chooses subsequent exploitation times inverse proportionally to the nonparametric estimation rate.
\begin{corollary}\label{coro:expl}
Let $a_i=2^i$ and $b_i = a_i/\Psi_{d,\bm{\beta}}(a_i)$ for $i \in \N$. Given the assumptions of Theorem \ref{theo:expl} it holds 
\[\frac{1}{T}\E\big[\tilde{C}_{0,T}\big] - J(D^\ast)  \lesssim \begin{cases} \big(\frac{(\log T)^2}{T}\big)^{\frac13}, &d= 2,\\ \big(\frac{\log T}{T}\big)^{\frac{\overline{\bm{\beta} + \bm{1}}}{3\overline{\bm{\beta} + \bm{1}} + d-2}}, &d\geq 3,\end{cases}\]
for the corresponding reflection strategy with exploration periods of length $a_i$ and exploitation periods of length $b_i$.
\end{corollary}
\begin{proof} 
Let $d \geq 3$ and denote $\alpha \coloneqq \overline{\bm{\beta} + \bm{1}}/(2\overline{\bm{\beta} + \bm{1}} + d-2)$. Let $T \in [\sum_{i=1}^n (a_i+b_i), \sum_{i=1}^{n+1} (a_i+b_i)]$. Then, $T \sim 2^{n(\alpha+1)}/n^\alpha$ and 
\[\sum_{i=1}^{n(T)} a_i + \sum_{i=1}^{n(T)} b_i\Psi_{d,\bm{\beta}}(a_i) \leq 2\sum_{i=1}^{n} a_i \sim 2^n,\]
which, by Theorem \ref{theo:expl} yields 
\[\frac{1}{T}\E[\tilde{C}_{0,T}] - J(D^\ast) \lesssim (n2^{-n})^\alpha = (n n^\alpha 2^{-n(\alpha +1)})^{\alpha/(\alpha+1)} \sim (\log T/T)^{\alpha/(\alpha+1)} = (\log T/T)^{\frac{\overline{\bm{\beta} + \bm{1}}}{3\overline{\bm{\beta} + \bm{1}} + d-2}}. \] 
The claim for $d=2$ is proved analogously.
\end{proof}
The rate loss of the strategy's regret per time unit relative to the static regret from Proposition \ref{prop:stat_regret} provides a natural analogue to the regret bounds from \cite[Theorem 2.5]{christensen21_alt} in the one-dimensional case, even though the construction of the strategies substantially differs. 
Let us also remark that \textit{doubling tricks} in the strategy design that make sure that the number of episodes at time $T$ is of order $\log T$, are commonly encountered in algorithms with verifiable optimal regret rates for undiscounted reinforcement learning problems. For instance,  the popular UCRL2 algorithm in  \cite{auer09} recomputes policies as soon as the occurrence of a state-action pair has doubled. Such strategies have the drawback of choosing suboptimal policies for arbitrarily long periods of time (here, not reflecting at all), see \cite{burnetas97}. Recently, \cite{boone23} have proposed the \textit{regret of exploration} as an appropriate measure to capture such inefficiencies. It is an interesting and challenging question for future work to adapt reflection strategies based on the nonparametric approach advocated in this paper to such finer-grained performance measures.

\begin{appendix}
\section{Remaining proofs} \label{sec:app}
\begin{proof}[Proof of Lemma \ref{lem:inv_dens}]
We show this using \cite[Theorem 1]{kang14}.
Since we consider only smooth sets $D$ and normal reflections, the set of test functions $\mathcal H$ as defined in \cite{kang14} simplifies to
\[
\mathcal H
=\{f\in \mathcal{C}^2_c(\overline{D})\oplus \R : \langle n(x), \nabla f(x)\rangle\ge0\text{ for }x\in\partial D\}.
\]
Letting $\pi(\mathrm{d}x)=\rho_D(x)\mathrm{d}x$, we note $\pi(\partial D)=0$, and thus it suffices to check
\[
\int_DAf(x)\, \pi(\diff{x})\le 0,\quad f\in \mathcal H,
\]
where $A$ is the differential operator given by 
\begin{equation}\label{eq:diff_op}
Af(x)
=-\langle \nabla V(x), \nabla f(x)\rangle+\Delta f(x),\quad f\in \mathcal{C}^2(\R^d), x\in\R^d.
\end{equation}
To this end, let $f\in \mathcal H$ be given, and we find using the divergence theorem
\begin{align*}
\int_{D}Af(x)\, \pi_D(\diff{x})
&=\int_{D} -\langle \nabla V(x),\nabla f(x)\rangle\rho_D(x)\diff{x}+\int_{D}\Delta f(x)\rho_D(x)\diff{x} \\
&=\int_D \langle\nabla f(x),\nabla \rho_D(x) \rangle \diff{x} +\int_{D}\Delta f(x)\rho_D(x)\diff{x}   \\
&=\Big(-\int_{\partial D}\langle \nabla f(x),n(x)\rangle\rho_D(x)\, \mathcal{H}^{d-1}(\diff{x})-\int_D\Delta f(x)\rho_D(x)\diff{x}\Big)\\ 
&\qquad+\int_{D}\Delta f(x)\rho_D(x)\diff{x} \\
&=-\int_{\partial D}\langle \nabla f(x), n(x)\rangle \rho_D(x)\, \mathcal{H}^{d-1}(\diff{x}),
\end{align*}
where $\mathcal H^{d-1}$ denotes the $(d-1)$-dimensional Hausdorff measure.
Since $f\in \mathcal H$, we have $\langle \nabla f, n\rangle\ge0$, and hence $\int_DAf(x)\, \pi_D(\mathrm{d}x)\le 0$ as desired.
\end{proof}

\begin{proof}[Proof of Lemma \ref{lem:erg}]
Since $p_t^D$ is continuous on $\overline{D}^2$, the process $X^D$ is a strong Feller process and, if we denote by $\lebesgue_D$ the restriction of the Lebesgue measure to $\overline{D}$, \eqref{eq:minor} ensures that  
\begin{equation} \label{eq:doeblin}
P_1^D(x,\cdot) \geq \delta \lebesgue_D, \quad x \in \overline{D}.
\end{equation} 
Consequently, the compact state space $\overline{D}$ is small and hence petite in the sense of \cite{MeynTweedie1993}.
 Using \cite[Theorem 1.1]{mt92}, it follows that $X^D$ is Harris recurrent.
 This implies that $X^D$ has a unique invariant distribution $\pi_D$, s.t.\ $\pi_D(\diff{x}) = \rho_D(x) \diff{x}$ follows from Lemma \ref{lem:inv_dens}. Moreover, the Doeblin recurrence property \eqref{eq:doeblin} implies by \cite[Theorem 16.2.4]{mt09} uniform ergodicity of the $1$-skeleton with explicit constants, that is
\[\lVert P_n^D(x,\cdot) - \pi_D \rVert_{\mathrm{TV}} \leq 2\mathrm{e}^{\log(1- \delta \lebesgue(D))  n}, \quad x \in \overline{D}.
\] 
By stationarity of $\pi_D$ it is easy to see that $\lVert P_t^D(x,\cdot) - \pi_D \rVert_{\mathrm{TV}}$ is decreasing in $t$, whence,
\[\lVert P_t^D(x,\cdot) - \pi_D \rVert_{\mathrm{TV}}\leq \lVert P_{\lfloor t \rfloor}^D(x,\cdot) - \pi_D \rVert_{\mathrm{TV}} \leq 2\mathrm{e}^{\log(1-\delta\lebesgue(D)) \lfloor t \rfloor} \leq \tfrac{2}{1-\delta\lebesgue(D)} \mathrm{e}^{\log(1-\delta\lebesgue(D)) t}.
\] 
\end{proof} 
\begin{proof}[Proof of Corollary \ref{coro:erg}]
Due to the exponential ergodicity of $X^D$, it follows from \cite[Corollary 3.3]{dex22} that 
\begin{equation}\label{eq:erg1}
\frac{1}{t}\E^{\pi_D}\Big[\Big\lvert \int_0^t (h(X_s^D) - \pi_D(h)) \diff{s}\Big\rvert\Big] \lesssim \frac{\lVert h \rVert_{L^\infty(\overline{D})}}{\sqrt{t}}.
\end{equation}
To get the result for deterministic initial condition $X^D_0 = x \in \overline{D}$, we note that 
for $t \geq 1$ and $\tilde{h} \coloneqq h - \pi_D(h)$ and $g_t(y) \coloneqq t^{-1}\E^y[\lvert \int_0^{t - \sqrt{t}} h(X_s^D) \diff{s}\rvert]$, we obtain
\begin{equation}\label{eq:erg2}
\begin{split}
&\Big\lvert \E^x\Big[\Big\lvert\frac{1}{t}\int_0^t \tilde{h}(X_s^D) \diff{s} \Big\rvert \Big] - \E^{\pi_D}\Big[\Big\lvert\frac{1}{t}\int_0^t \tilde{h}(X_s^D) \diff{s} \Big\rvert \Big]\Big\rvert\\
&\,\leq 2\lVert \tilde{h} \rVert_{L^\infty(\overline{D})} \frac{1}{\sqrt{t}} +  \Big\lvert \E^x\Big[\Big\lvert\frac{1}{t}\int_0^{t - \sqrt{t}} \tilde{h}(X_{s+\sqrt{t}}^D) \diff{s} \Big\rvert \Big] - \E^{\pi_D}\Big[\Big\lvert\frac{1}{t}\int_0^{t -\sqrt{t}} \tilde{h}(X_{s+ \sqrt{t}}^D) \diff{s} \Big\rvert \Big]\Big\rvert\\
&\,= 2\lVert \tilde{h} \rVert_{L^\infty(\overline{D})}\frac{1}{\sqrt{t}} + \big\lvert \E^x[g_t(X_{\sqrt{t}})] - \pi_D(g_t) \big\rvert \\
&\,\leq 2\lVert \tilde{h} \rVert_{L^\infty(\overline{D})}\Big(\frac{1}{\sqrt{t}} + \lVert P^D_{t^{1/2}}(x,\cdot) - \pi_D \rVert_{\mathrm{TV}}\Big)\\
&\lesssim \lVert h \rVert_{L^\infty(\overline{D})} \frac{1}{\sqrt{t}},
\end{split}
\end{equation}
where we used the Markov property and stationarity of $\pi_D$ for the second line, $\lVert g_t \rVert_\infty \leq \lVert \tilde{h} \rVert_{L^\infty(\overline{D})}$ for the third line and Lemma \ref{lem:erg} for the last line. The claim now follows from combining \eqref{eq:erg1} and \eqref{eq:erg2} with the triangle inequality.
\end{proof}
\end{appendix}

\paragraph*{Acknowledgements}
AHT is grateful for financial funding by Villum Synergy Programme, project no.\ 50099. LT gratefully acknowledges financial support of Carlsberg Foundation Young Researcher Fellowship grant CF20-0640.

\printbibliography

\end{document}